\documentclass{amsart}
 \usepackage[svgnames]{xcolor}
\usepackage{amsmath}

\usepackage{thmtools}
\usepackage{stmaryrd}
\SetSymbolFont{stmry}{bold}{U}{stmry}{m}{n}
\usepackage[english]{babel}
\usepackage[utf8]{inputenc}
\usepackage[T1]{fontenc}   
\usepackage{enumitem}
\usepackage{mathrsfs}
\usepackage{bm}
\usepackage{stmaryrd}
\usepackage{amssymb,amsmath,amsthm}
\usepackage{xstring}
\usepackage{mathtools}
\usepackage{tikz-cd}
\usepackage{csquotes}
\usepackage{physics,derivative}

\newtheorem{theorem}{Theorem}[section]
\newtheorem{lemma}[theorem]{Lemma}

\theoremstyle{definition}
\newtheorem{definition}[theorem]{Definition}
\theoremstyle{plain}
\newtheorem{proposition}[theorem]{Proposition}
\newtheorem{corollary}[theorem]{Corollary}
\newtheorem*{theorem*}{Theorem}
 
\theoremstyle{remark}
\newtheorem{remark}[theorem]{Remark}

\newtheorem{example}[theorem]{Example}
\usepackage[style=numeric,backend=biber, sorting=nyt, giveninits=false, eprint=true,isbn=false,doi=true,url=false]{biblatex}

\usepackage{accents}   

\title[Liouville Rigidity for Hessian Equations]{Liouville Rigidity for Real and Complex Degenerate  Hessian Equations}

\author{Hao Fang}
\email{hao-fang@uiowa.edu}
\address{Department of Mathematics, University of Iowa, Iowa City, IA 52246,
USA}
\author{Biao Ma}
\email{bma@math.ecnu.edu.cn}
\address{School of Mathematical Sciences, East China Normal University, 500
Dongchuan road, Shanghai, China.}
\thanks{H. F.'s work is partially supported by a Simons Foundation mathematics
collaboration grant and an NSF-RTG grant. B. M. is partially supported
by NSFC grant (No. 12471052). J. Y. Wu is partially supported by NSFC grant (No. 12371207).}

\author{Jinyang Wu}
\email{wujy3@shanghaitech.edu.cn}
\address{Institute of Mathematical Sciences (IMS),
ShanghaiTech University, 393 Middle Huaxia Road, Shanghai, China
}

\usepackage[linktocpage]{hyperref}
\hypersetup{
	colorlinks=true,
	linkcolor=blue,
	filecolor=blue,
	citecolor = black,      
	urlcolor=cyan,
}

\begin{document}
\begin{abstract}
We prove Liouville rigidity theorems for translation-invariant real and
complex Hessian equations in the viscosity sense, where the PDE is
encoded by an admissible set $\mathcal{A}$.
The main structural notion is \emph{Liouville admissibility}, a recursive
geometric condition requiring each quotient set  to be either boundary compatible or to fall into a terminal class. Our main theorem
states that every bounded, globally \(C^{0,\alpha}\)
entire viscosity solution of 
\[
\mathrm{Hess}_{\mathbb F}u\in\partial\mathcal{A}
\]
     is constant \emph{if and only if} $\mathcal{A}$ is Liouville admissible; thus   the Liouville-type property is characterized as a geometric property of the admissible set.

A central class of examples arises from polarizations of univariate
G{\aa}rding polynomials satisfying the monotone root sequence condition,
producing mixed elementary-symmetric admissible sets and recovering the
standard \(k\)-Hessian equations as monomial cases. The framework also
allows anisotropic constructions, including linear pullbacks and
intersections of admissible sets.

\end{abstract}
	\maketitle
\newcommand{\A}{\mathcal{A}}
\newcommand{\Pol}{\Pi^\uparrow}
\newcommand{\pd}{\Gamma^n_+}
\newcommand{\psd}{\overline{\Gamma^n_+}}

\newcommand{\Logcone}{\mathcal{L}}
\newcommand{\x}{\mathbf x}
\newcommand{\R}{\mathbb{R}}
\newcommand{\C}{\mathbb{C}}
\newcommand{\HH}{\mathcal{H}}
\newcommand{\SSS}{\mathcal{S}}
\newcommand{\HessR}{\nabla^2_{\mathbb{R}}}
\newcommand{\HessC}{\nabla\bar{\nabla}}
\newcommand{\LLL}{{\mathcal{L}}}
\newcommand{\Hess}{\operatorname{Hess}}
\newcommand{\Q}{\mathcal Q}
\newcommand{\<}{\langle}
\renewcommand{\>}{\rangle}
\newcommand{\TT}{\mathcal T}

\providecommand{\Cone}{\mathrm{Cone}}

%\input{ComplexKrylovtypeclean}
%\input{Draft1}
%\input{Draft2}
%input{Draft3-Temp}
%\input{Draft3-main}
%\input{Draft4-intro}
%\input{Draft4-Temp}
\section{Introduction}

Liouville theorems are rigidity results for entire solutions of
elliptic equations. The classical example states that a
harmonic function on a Euclidean space bounded from above is constant. For fully
nonlinear degenerate Hessian equations, the corresponding rigidity is more subtle. In such problems, one must specify an admissible set serving as the defining domain of the elliptic operator. The Liouville property is then determined not only by the elliptic equation itself, but also by the geometry of the chosen admissible set. 

The purpose of this paper is to prove Liouville rigidity theorems for
translation-invariant real and complex Hessian equations. Let
\(\mathbb F\in\{\mathbb R,\mathbb C\}\), and let \(V\) be a finite-dimensional vector space over \(\mathbb F\). We write
\(\mathcal Q(V)\) for the space of real symmetric bilinear forms when
\(\mathbb F=\mathbb R\), and for the space of Hermitian forms when
\(\mathbb F=\mathbb C\). Given an admissible set
\(\A\subset\mathcal Q(V)\), we consider the boundary equation
\begin{equation}\label{PDE}
\Hess_{\mathbb F}u\in\Sigma:=\partial\A.
\end{equation}
Here, \(\Hess_{\mathbb F}\) denotes the real Hessian in the real case
and the complex Hessian in the complex case.

Roughly, an admissible set \(\A\) is a closed convex elliptic subset of
\(\mathcal Q(V)\) that contains the origin and satisfies a positive trace
condition.
Equation \eqref{PDE} is understood in the
viscosity sense.
After a metric is fixed, \(\A\) is called spectral if it is invariant
under the orthogonal group in the real case or the unitary group in the complex case, respectively. Equivalently, membership in \(\A\) depends
only on the unordered eigenvalues.
The precise definitions of admissible sets and viscosity solutions are given in~\autoref{sec:real-complex-hessian}. 

Our work is motivated by geometric PDEs, where the appropriate notion of
admissibility often carries geometric meaning. This is evident in the
classical real and complex Monge--Amp\`ere equations and in
\(k\)-Hessian equations. The admissible region also plays an important
role in mixed symmetric equations and in fully nonlinear equations
arising in complex geometry. See, for example, the work of Tosatti--Weinkove
\cite{MR3600038}.
The present paper provides an abstract framework in which admissible sets
and the resulting Liouville rigidity can be studied in a unified manner.

To establish Liouville theorems, we introduce the notion of Liouville admissible sets, which involves the geometric operation of quotient reduction. For every nonzero proper \(\mathbb F\)-linear subspace \(N\subset V\), we obtain a reduced admissible set
\[
\A_{/N}\subset \mathcal Q(V/N).
\]
The Liouville admissibility condition requires that either the boundary equation reduces compatibly or the reduced set belongs to a terminal set, where every bounded-above subsolution is constant. This recursive structure is designed to support an induction argument on dimension. See \autoref{sec:real-complex-hessian} and \autoref{Appendix1} for precise definitions and details.

Our first main result is the following geometric characterization of Liouville rigidity.

\begin{theorem}\label{thm:intro-main}

Let \(0<\alpha\le1\). Let \(V\) be a real or complex linear vector space, and let
\(\mathcal Q(V)\) denote the space of symmetric bilinear forms on \(V\) in the real case and Hermitian forms on \(V\) in the complex case. Let
\(\A\subset\mathcal Q(V)\) be admissible. Then the following are equivalent:
\begin{enumerate}
\item \(\A\) is Liouville admissible.

\item Every bounded globally \(C^{0,\alpha}\) entire viscosity solution of
\[
\Hess_{\mathbb F}u\in\Sigma,
\qquad
\Sigma=\partial\A,
\]
is constant.
\end{enumerate}
\end{theorem}

\autoref{thm:intro-main} identifies
Liouville rigidity as an intrinsic geometric property of the admissible
set.  It gives a
sharp equivalence between the analytic Liouville property for bounded
H\"older entire viscosity solutions and the recursive reduction
condition of Liouville admissibility.  In one direction, Liouville
admissibility provides the induction mechanism to prove rigidity; in
the other direction, failure of Liouville admissibility produces a
nonconstant bounded smooth solution after passing to a quotient,
mollifying, and pulling back.  Thus, the question whether
\eqref{PDE}
has nonconstant bounded entire solutions is reduced to the quotient
geometry of \(\A\), and the Liouville property is seen to depend on the
chosen admissible region rather than only on the algebraic boundary
equation.
 
The present work is related to three lines of existing work. 

First, in the complex setting, Dinew--Ko{\l}odziej~\cite{MR3636634} established a Liouville theorem for bounded entire maximal \(k\)-subharmonic functions on \(\mathbb C^n\) with bounded gradient. As an application, they obtained gradient estimates for complex \(k\)-Hessian equations on compact Kähler manifolds, completing a program initiated by Hou--Ma--Wu~\cite{MR2653687}. This Liouville theorem was subsequently generalized by \citeauthor{MR3807322}~\cite{MR3807322}. More recently, Fu--Yau--Zhang~\cite{fu2026criticallyzequationkahler} studied a broader class of complex Hessian equations, proved the corresponding Liouville theorem, and derived the gradient estimate for the critical LYZ equation.

Second, since the 2000s, Y. Y. Li has investigated viscosity solutions of fully nonlinear differential inclusions arising in conformal geometry~\cite{MR2547976}. In that and subsequent works
\cite{MR2547976,MR1988895,MR3664220,MR4181003,chu2024liouvilletheoremsconformallyinvariant},
the corresponding geometric PDEs are often formulated with a symmetric
elliptic function applied to the eigenvalues of the Schouten tensor;
these works establish Liouville, quantitative Liouville, comparison, and
blow-up results for conformally invariant fully nonlinear equations.

Finally, our work is close in spirit to the Harvey--Lawson theory of
subequations and Dirichlet duality, where fully nonlinear degenerate
elliptic equations are represented by closed constraint sets in the jet
bundle and studied through subharmonicity, duality, monotonicity, and
restriction \cite{MR2487853,MR2844439,MR3330548}.

The study of fully nonlinear Hessian equations goes back to the foundational works of Caffarelli--Nirenberg--Spruck and Kohn
\cite{MR739925,MR780073,MR806416}. Since then, an extensive literature
has developed; we mention only a small and necessarily incomplete
selection of references
\cite{MR3703453,MR3661790,MR4278951,MR3600038,MR1368245,MR1275451, MR1634570, MR1726702, MR1923626, MR3245188, MR3091809,MR3385185,MR3284698,MR4545123,MR2653687,MR1867613,MR2663711,MR3919441,MR3169787,MR4531386,MR2644763,MR3464054,MR1963687}.

The formulation in \eqref{PDE} treats real and complex Hessian
equations in a unified way, but the corresponding terminal models are
different. In the complex case, the basic spectral terminal model is the
positive Hermitian cone, reflecting the classical Liouville theorem for
bounded-above entire plurisubharmonic functions. In the real case, there
is more flexibility: besides the positive semidefinite cone, one has
larger terminal models, most notably the real logarithmic cone. In
\autoref{Appendix1} we identify them as the maximal spectral terminal cones in
both settings.

In the complex setting, the viscosity formulation used here is
compatible with the pluripotential viewpoint for complex Hessian
equations but is designed to apply to general admissible sets under
weaker regularity assumptions. In particular, our
Liouville theorem 
contains the Dinew--Ko{\l}odziej theorem \cite{MR3636634} as a consequence in the
complex Hessian setting.

The proof of the main theorem combines viscosity theory with the
linear-algebraic geometry of admissible sets.  Its novelty lies in this
intrinsic geometric viewpoint, while the technical arguments are
inspired by earlier works of \cite{MR3636634} and \cite{MR3807322}, with subtle improvements to work with our new setup.

A central source of examples for Theorem~\ref{thm:intro-main} is the spectral case. 

After fixing a Euclidean or Hermitian metric, a spectral admissible set is an admissible set that is invariant under the orthogonal or unitary group,
respectively. Hence a spectral admissible set is determined by an eigenvalue region in $\R^n$. See \autoref{sec:geometry-liouville-admissible} for the proof of \autoref{thm:spectralimpliesliouville}.
\begin{theorem}\label{thm:spectralimpliesliouville}
    A spectral admissible set is Liouville admissible. 
\end{theorem}
The admissible sets for the Laplace equation, the Monge--Amp\`ere equation, and  the $k$-Hessian equations are fundamental examples.

An important example comes from the recent paper of Fu--Yau--Zhang \cite{fu2026criticallyzequationkahler}. In their work, Fu--Yau--Zhang solved the critical LYZ equation in K\"ahler geometry. To establish a gradient estimate, they introduced the equation
\begin{equation}
\label{FYZ}
\sigma_{n-1}(\Hess_{\mathbb C}u)
+
\sigma_n(\Hess_{\mathbb C}u)
=
0.
\end{equation}
They showed that every bounded, 
entire $C^{0,1}$, $(n-1)$-subharmonic function satisfying \eqref{FYZ} must be constant. To the best of our knowledge, this is the first Liouville theorem for an inhomogeneous equation of this type. 
Our \autoref{thm:intro-main} yields an alternative proof of their result.

Our next result focuses on the construction of spectral Liouville admissible sets using the theory of G{\aa}rding and ideal G{\aa}rding polynomials developed in recent work of the first and second named authors \cite{fang2026gaardingpolynomials,fang2026idealgaardingpolynomials}. This theory strictly extends the theory of multivariate real-stable polynomials and, from the PDE point of view, provides a systematic method for producing convex elliptic regions in the eigenvalue space. We use the following explicit one-variable input.

Let
\begin{equation}\label{polynomial}
    f(t)=\sum_{k=1}^{d}a_k t^k,
    \qquad
    a_k\ge0,
    \qquad
    a_d>0,
\end{equation}
have no constant term. Assume that \(f\) satisfies the monotone root
sequence (MRS) condition: every derivative \(f^{(j)}\), \(0\le j\le d-1\),
has a real root, and the largest real roots $r(f^{(j)})$ satisfy
\begin{equation}
\label{mrs-intro}
    0=r(f)\ge r(f')\ge\cdots\ge r(f^{(d-1)}).
\end{equation}
Such polynomials were first introduced as right-Noetherian polynomials
by \citeauthor{MR4604837} \autocites{MR4604837,LIN2024110398}, and were later
identified as univariate G{\aa}rding polynomials in
\autocite{fang2026gaardingpolynomials}. For \(n\ge d\), let
\(\Pol_n(f)\) denote the \(n\)-variable multi-affine polarization of
\(f\), normalized by
\[
    \Pol_n(f)(t,\ldots,t)=f(t).
\]

\begin{theorem} 
\label{thm:intro-example}
Let \(f\) satisfy \eqref{polynomial} and \eqref{mrs-intro}. Then, for every \(n\ge d\), the
polarization \(g=\Pol_n(f)\) is G{\aa}rding  and has a G{\aa}rding component whose spectral lift is Liouville admissible. Consequently, the
associated spectral Hessian boundary equation satisfies the conclusion of the Liouville rigidity  in \autoref{thm:intro-main}.
\end{theorem}

Therefore,  the explicit MRS condition \eqref{mrs-intro} gives an algebraic
recipe for producing Liouville-rigid spectral Hessian equations. It
recovers the usual \(k\)-Hessian branches from \(f(t)=t^k\), and gives
mixed elementary-symmetric components from polynomials with more than one
term, such as \eqref{FYZ}.

\begin{remark}
In the spectral multi-affine polynomial setting, an admissible convex elliptic region
cannot arise from an arbitrary polynomial. It has to be the result of a polarization of  a one-variable polynomial satisfying the
MRS condition by \autocite{fang2026gaardingpolynomials}. Hence, \autoref{thm:intro-example}   provides the full class of such Liouville-admissible spectral examples.
\end{remark}

The spectral theory is only one part of the picture. Liouville
admissibility is stable under intersections and actions of the general linear group \(GL_{\mathbb F}(V)\). As a consequence,
our framework also contains many non-isotropic examples.

The paper is organized as follows. In
\autoref{sec:real-complex-hessian} we introduce the real and
complex Hessian framework, admissible sets, terminal sets, Liouville
admissibility, and the precise statement of the main theorem.
\autoref{sec:geometry-liouville-admissible} develops the geometry
of Liouville admissible sets, including quotient reductions, stability
properties, and the boundary dichotomy. In
\autoref{sec:applications-univariate-garding} we construct
spectral Liouville admissible sets from polarized univariate
G{\aa}rding polynomials. The remaining sections are devoted to the
proof of the main theorem: first, the analytic preliminaries, then the
basic properties of viscosity solutions, and finally the
dimension-reduction argument.

\subsection*{Acknowledgements}
The second named author thanks Dekai Zhang for helpful discussions.

\section{Real And Complex Hessian Equations}
\label{sec:real-complex-hessian}
In this section, we fix notation for our work, introduce key constructions and definitions, and state our main theorem.

We consider the real and complex cases simultaneously. Let
\(\mathbb F\in\{\mathbb R,\mathbb C\}\), and let \(V\) be an
\(n\)-dimensional vector space over \(\mathbb F\).

We define
\begin{equation}
\mathcal Q(V)
=
\begin{cases}
\SSS(V), & \mathbb F=\mathbb R,\\
\mathcal H(V), & \mathbb F=\mathbb C,
\end{cases}
\end{equation}
where \(\SSS(V)\) denotes the space of real symmetric bilinear forms on
\(V\), and \(\mathcal H(V)\) denotes the space of Hermitian forms on
\(V\). We write \(\mathcal Q_+(V)\) and \(\mathcal Q_{++}(V)\) for the
cones of positive semidefinite and positive definite forms, respectively.

For \(u\in C^\infty(\Omega)\), where \(\Omega\subset V\), define
\begin{equation}
\Hess_{\mathbb F}u
=
\begin{cases}
\HessR u= (u_{ij}), & \mathbb F=\mathbb R,\\
\HessC u=(u_{i\bar j}), & \mathbb F=\mathbb C.
\end{cases}
\end{equation}
Thus \(\Hess_{\mathbb F}u\in\mathcal Q(V)\). For example, in complex coordinates,
if
\[
\mathbf v=(v^1,\cdots,v^n)
\qquad
\mathbf w=(w^1,\cdots,w^n),
\]
then
\[
\HessC u(x)(\mathbf v,\mathbf w)
=
\sum_{i,j=1}^n
u_{i\bar j}(x)\,v^i\,\overline{w}^j.
\]

For \(g\in\mathcal Q_{++}(V)\) and \(A\in\mathcal Q(V)\), define
\[
\Tr_g A
=
\sum_{\alpha=1}^{n}A(e_\alpha,e_\alpha),
\]
where \(\{e_1,\ldots,e_n\}\) is any \(g\)-orthonormal basis of \(V\).
This quantity is independent of the choice of \(g\)-orthonormal basis.

We introduce admissible sets.

\begin{definition}\label{def:admissible}
A closed subset
\(
\A\subset\mathcal Q(V)
\)
is called \emph{admissible} if:
\begin{enumerate}[label=\textup{(A\arabic*)}]
\item \(0\in\A\);\label{Cond:origin}

\item \(\A+\mathcal Q_+(V)\subset\A\); \label{Cond:up}

\item \(\A\) is convex; \label{Cond:convex}

\item there exists \(h\in\mathcal Q_{++}(V)\) such that
\[
\Tr_h A\ge0
\qquad
\text{for all }A\in\A.
\]\label{Cond:subharmonic}
\end{enumerate}
\end{definition}

We also define
\(
\Sigma=\partial\A.
\)
In this paper, we consider the following  partial differential subequation:
\begin{equation}
\Hess_{\mathbb F}u\in\A,\label{subequation}
\end{equation}
and equation:
\begin{equation}
\Hess_{\mathbb F}u\in\Sigma.\label{equation}
\end{equation}

We recall the definitions of  viscosity solutions and H\"{o}lder continuity.
\begin{definition}
A function \(u:V\to\mathbb R\) is a viscosity subsolution of \eqref{equation}
if \(u\) is upper semicontinuous and, whenever
\(u-\phi\) has a local maximum at \(\mathbf x_0\in V\) for some
\(\phi\in C^2(V)\), then
\[
\Hess_{\mathbb F}\phi(\mathbf x_0)\in\A.
\]

A viscosity supersolution is defined similarly using lower
semicontinuity, local minima, and \(\overline{\A^{\,c}}\). A viscosity
solution is both a viscosity subsolution and a viscosity supersolution.
\end{definition}
\begin{definition}
For $0 < \alpha \leq 1$, a function $u \colon V \to \mathbb R$ is said to be uniformly $\alpha$-H\"{o}lder continuous if for some norm $\vert\cdot\vert_V$ on $V$,
\begin{equation}
    [u]_{0,\alpha; V}
    =
    \sup_{x \neq y}
    \frac{|u(x) - u(y)|}{|x - y|_V^\alpha}
    < \infty .
\end{equation}
In particular, the finiteness of $[u]_{0,\alpha; V}$ is independent of the norm chosen since all norms on a finite-dimensional vector space are equivalent.
\end{definition}

Fix $V$ as before. Given an admissible set $\A$ satisfying Definition~\ref{def:admissible}, we consider the bounded uniformly $\alpha$-H\"{o}lder continuous solutions of (\ref{equation}). The following example shows that a Liouville theorem for general admissible sets does not hold.
\begin{example}
    Let $W=\mathbb{R}^4$ and $\mathcal{A}=\{A\in\SSS(W):A_{11}\ge0,\ A_{22}+A_{33}+A_{44}\ge0\}$.
Setting $x'=(x_2,x_3,x_4)$ and
\[
u(x_1,x')=\begin{cases}-1,&|x'|\le1,\\-|x'|^{-1},&|x'|>1,\end{cases}
\]
we have $u_{11}\equiv0$, so $D^2u\in\Sigma=\{A\in {\A}:A_{11}(A_{22}+A_{33}+A_{44})=0\}$,
i.e., $u$ is a bounded non-constant solution of $D^2u\in\Sigma$.
 
\end{example}

We explore the following good situation, in which the upper bound of subsolutions ensures the Liouville type results.

\begin{definition}[Terminal set] \label{def:liouvillecone}
An admissible set
\(
\mathcal T \subset\mathcal Q(W)
\)
is called a \emph{terminal set} if every continuous viscosity subsolution of
\[
\Hess_{\mathbb F}u\in\mathcal T
\]
which is bounded above on \(W\) is constant.
\end{definition}

\begin{proposition}
\label{prop:positive-cone-terminal}

Let \(V\) be a finite-dimensional real Euclidean or complex Hermitian
vector space. Then the positive cone \(\mathcal Q_+(V)\) is terminal.

\end{proposition}

\begin{proof}By the restriction theorem for subequations
\autocite{HarveyLawsonRestriction},
the restriction of \(u\) to every \(\mathbb F\)-affine line
\(\ell\subset V\) is a viscosity
\(\mathcal Q_+(\ell)\)-subsolution. Therefore, \(u\)
restricts itself  to a convex function on every real affine line (respectively,
a subharmonic function on every complex affine line). Since these
restrictions are bounded above, they are constant. Therefore, \(u\) is
constant.
\end{proof}

\autoref{prop:positive-cone-terminal} establishes the existence of terminal sets, which are a special type of admissible set. In general, terminal sets are not unique for a given $V$. Their geometry will be discussed in \autoref{Appendix1}.

To connect the notions of general admissible sets and terminal sets, we introduce the following reduction procedure, which is our central geometric construction.

For \(A\in\mathcal Q(V)\), define
\begin{equation}
\ker A
=
\{\mathbf v\in V:A(\mathbf v,\mathbf w)=0
\text{ for all }\mathbf w\in V\}.
\end{equation}

Let \(N\subset V\) be a proper \(\mathbb F\)-linear subspace, and let
\(
q_N:V\to V/N
\)
be the quotient map. If \(N\subset\ker A\), then \(A\) descends
uniquely to 
\(
A_{/N}\in\mathcal Q(V/N)
\)
such that
\[
A=q_N^*A_{/N}.
\]
For any \(S\subset\mathcal Q(V)\), define
\[
S_{/N}
=
\{B\in\mathcal Q(V/N):q_N^*B\in S\}.
\]
Thus,
\[
\A_{/N}
=
\{B:q_N^*B\in\A\},
\qquad
\Sigma_{/N}
=
\{B:q_N^*B\in\Sigma\}.
\]

Note that \(\Sigma_{/N}\) is defined as the reduction of \(\Sigma\), not
a priori as \(\partial(\A_{/N})\).

\begin{lemma}\label{lem:admissibilityofquotient}
If \(\A\) is admissible, then \(\A_{/N}\) is admissible for every
\(\mathbb F\)-linear subspace \(N\subset V\).
\end{lemma}

\begin{proof}
The properties~\ref{Cond:origin}, \ref{Cond:up} and \ref{Cond:convex} of $\A_{/N}$
follow immediately from the corresponding properties of \(\A\).

If \(h\in\mathcal Q_{++}(V)\) satisfies
\(
\Tr_h A\ge0\ 
\text{for all }A\in\A,\)
then $h$ induces a natural isomorphism between $V/N$ and $(N^\perp, h^\perp)$, the $h$-orthogonal complement of $N\subset V$ with restricted metric. Therefore, there exists a canonical
\(
h_{/N}\in\mathcal Q_{++}(V/N)
\)
such that
\begin{equation}\label{00}
    \Tr_{h_{/N}}B
=
\Tr_h(q_N^*B)
\end{equation}
for all \(B\in\mathcal Q(V/N)\). Since $q^*_N(\A_{/N})\subset\A$, \eqref{00} implies that
\(
\Tr_{h_{/N}}B\ge0
\
\text{for all }B\in\A_{/N}.
\)
\end{proof}

We now introduce a geometric notion that encodes the Liouville property for the Hessian equation.
 
\begin{definition}\label{defn:Liouville-admissibility}
An admissible set
\(\A\subset\mathcal Q(V)\)
is called \emph{Liouville admissible} if, for every nonzero proper
\(\mathbb F\)-linear subspace \(N\subset V\), one has
\[
\left\{
\begin{array}{ll}
\text{either} & \Sigma_{/N}=\partial_{V/N}(\A_{/N}),\\[4pt]
\text{or} & \A_{/N}\text{ is contained in a terminal set in }\mathcal Q(V/N).
\end{array}
\right.
\]
\end{definition}

Thus, for a Liouville admissible set, every quotient reduction of it either preserves the boundary equation structure or collapses into a terminal set.

Finally, we present our main result, which is an alternative form of~\autoref{thm:intro-main}.

\begin{theorem}
\label{thm:main}The notation is set as above.
Let \(0<\alpha\le 1\), and let
\(\A\subset \mathcal Q(V)\) be admissible. Then \(\A\) is Liouville admissible if and only if
every bounded globally \(C^{0,\alpha}\) entire viscosity solution of
\(\Hess_{\mathbb F}u\in\Sigma\) is constant.
\end{theorem}

The proof of Theorem~\ref{thm:main} will be given in the final sections, after we have analyzed the geometry of the sets involved.
\section{Geometry Of Admissible Sets}
\label{sec:geometry-liouville-admissible}

We record several geometric properties of admissible sets and of two distinguished classes: terminal sets and Liouville admissible sets. We also investigate their properties under the spectral assumption.

By definition,
\begin{equation}
\label{eq:terminal-liouville-admissible}
\{\text{terminal sets}\}
\subsetneq
\{\text{Liouville admissible sets}\}
\subsetneq
\{\text{admissible sets}\}.
\end{equation}
 
We prove stability of these classes under finite intersections, under
the natural action of the general linear group, and under quotient
reductions.

\subsection{Admissible Sets} We first examine admissible sets.

\begin{lemma}
\label{lem:admissible-stability}

Let \(V\) be a finite-dimensional real Euclidean or complex Hermitian
vector space as in Section~\ref{sec:real-complex-hessian}.
\begin{enumerate}

\item\label{item:admissible-intersection}
If \(\A_1,\ldots,\A_m\subset\mathcal Q(V)\) are admissible, then their intersection
\[
\A:=\bigcap_{j=1}^m\A_j
\]
is admissible.

\item\label{item:admissible-GL}
If \(\A\subset\mathcal Q(V)\) is admissible and
\(\varphi\in GL_{\mathbb F}(V)\), then
\[
\varphi^*\A
:=
\{\varphi^*A: A\in\A\},
\qquad
\varphi^*A(v,w)=A(\varphi v,\varphi w),
\]
is admissible.

\item\label{item:admissible-quotient}
If \(\A\subset\mathcal Q(V)\) is admissible and \(N\subset V\) is an
\(\mathbb F\)-linear subspace, then the quotient reduction
\[
\A_{/N}\subset\mathcal Q(V/N)
\]
is admissible.

\end{enumerate}
\end{lemma}

\begin{proof}
Part~\ref{item:admissible-intersection} follows immediately from
Definition~\ref{def:admissible}. Part~\ref{item:admissible-GL}
follows from Definition~\ref{def:admissible} together with the identity
\[
\varphi^*\mathcal Q_+(V)=\mathcal Q_+(V).
\]
Part~\ref{item:admissible-quotient} is
Lemma~\ref{lem:admissibilityofquotient}.

\end{proof}

\subsection{Terminal Sets} Next, we consider terminal sets. We start with some basic properties. The proof of the following lemma is a direct check of the definitions, so we omit it here.

\begin{lemma}
\label{lem:terminal-stability}

Let \(V\) be a finite-dimensional real Euclidean or complex Hermitian
vector space.

\begin{enumerate}

\item\label{item:terminal-subset}
If \(\TT\subset\mathcal Q(V)\) is terminal and
\(
\A\subset\TT
\)
is admissible, then \(\A\) is terminal.

\item\label{item:terminal-GL}
If \(\TT\subset\mathcal Q(V)\) is terminal and
\(\varphi\in GL_{\mathbb F}(V)\), then
\(
\varphi^*\TT
\)
is terminal.

\end{enumerate}

\end{lemma}

For a proper subspace $N\subset V$, denote the quotient map \[q_N:V\to V/N.\] 
\begin{lemma}
\label{lem:pullback-quotient-subsolution}
Let \(\A\subset\mathcal Q(V)\) be admissible and let
\(N\subset V\) be a proper \(\mathbb F\)-linear subspace. If \(v\) is a
viscosity subsolution of
\[
\Hess_{\mathbb F}v\in\A_{/N},
\]
then \(\widetilde v:=v\circ q_N\) is a viscosity subsolution of
\[
\Hess_{\mathbb F}\widetilde v\in\A.
\]
\end{lemma}

\begin{proof}
Let \(\phi\in C^2(V)\) touch \(\widetilde v\) from above at
\(x_0\in V\). Fix \(G\in\mathcal Q_{++}(V)\). Since
\(q_N:V\to V/N\) is a linear submersion, the
Crandall--Ishii--Lions theorem \autocite[Theorem 3.2]{MR1118699} implies that, for every
\(\varepsilon>0\), there exists \(B_\varepsilon\in\A_{/N}\) such that
\[
q_N^*B_\varepsilon
\le
\Hess_{\mathbb F}\phi(x_0)+\varepsilon G
\]
as quadratic forms on \(V\). Equivalently, there exists some \(P_\varepsilon\in\mathcal Q_+(V)\) such that
\[
\Hess_{\mathbb F}\phi(x_0)+\varepsilon G
=
q_N^*B_\varepsilon+P_\varepsilon.
\]

Since \(B_\varepsilon\in\A_{/N}\), one has
\(q_N^*B_\varepsilon\in\A\). By ellipticity,
\(
q_N^*B_\varepsilon+P_\varepsilon\in\A.
\)
Hence,
\(
\Hess_{\mathbb F}\phi(x_0)+\varepsilon G\in\A.
\)
Letting \(\varepsilon\downarrow0\) and using that \(\A\) is closed, we
obtain
\[
\Hess_{\mathbb F}\phi(x_0)\in\A.
\]
Therefore, \(\widetilde v\) is an \(\A\)-subsolution.
\end{proof}

\begin{proposition}
\label{prop:liouville-quotient}
Let \(\LLL\subset\mathcal Q(V)\) be a terminal set and let
\(N\subset V\) be a proper \(\mathbb F\)-linear subspace. Then
\[
\LLL_{/N}
=
\{B\in\mathcal Q(V/N):q_N^*B\in\LLL\}
\]
is again a terminal set.
\end{proposition}

\begin{proof}
By Lemma~\ref{lem:admissibilityofquotient},
\(\LLL_{/N}\) is admissible. Let $v$
be a continuous bounded-above viscosity subsolution of \[
\Hess_{\mathbb F}v\in\LLL_{/N}.
\]
By \autoref{lem:pullback-quotient-subsolution},
\(
\widetilde v:=v\circ q_N
\)
is a continuous viscosity subsolution of
\(
\Hess_{\mathbb F}\widetilde v\in\LLL
\)
on \(V\). Since \(v\) is bounded above, so is \(\widetilde v\). Because
\(\LLL\) is a terminal set, \(\widetilde v\) is constant. Since
\(q_N\) is surjective, \(v\) is constant. Therefore, \(\LLL_{/N}\) is a
terminal set.
\end{proof}

\begin{lemma}
\label{lem:low-dimensional-admissible-terminal}

Let \(V\) be real with \(\dim_{\mathbb R}V\le2\), or complex
with \(\dim_{\mathbb C}V=1\). Then every admissible set
\[
\A\subset\mathcal Q(V)
\]
is a terminal set.

\end{lemma}

\begin{proof}

The zero-dimensional case is trivial. 

In real dimension $1$, $\mathcal{A}=\mathcal Q_+(V)=[0,+\infty)$. Entire \(\A\)-subsolutions are convex and hence must be constant whenever bounded above.

If $\dim_\mathbb{R}V=2$ or $\dim_\mathbb{C}V=1$, then any $\A$-subsolution is subharmonic with respect to some $h\in\mathcal{Q}_{++}(V)$. Hence, the result follows from the classical Liouville theorem in real dimension $2$. 
\end{proof}

\subsection{Liouville Admissible Sets}
\begin{lemma}
\label{lem:low-dimensional-liouville}
Let \(V\) be real with \(\dim_{\mathbb R}V\le3\), or complex with
\(\dim_{\mathbb C}V\le2\). Then every admissible set
\(
\A\subset\mathcal Q(V)
\)
is Liouville admissible.
\end{lemma}

\begin{proof}This is a direct consequence of \autoref{lem:low-dimensional-admissible-terminal}.
\end{proof}

\begin{lemma}
\label{lem:Liouvilleadmissibilityofquotient}
Let \(\A\subset\mathcal Q(V)\) be Liouville admissible, and let
\(N\subset V\) be a proper \(\mathbb F\)-linear subspace. Then
\(\A_{/N}\) is Liouville admissible. 
\end{lemma}
\begin{proof}
By \autoref{lem:admissibilityofquotient}, \(\A_{/N}\) is admissible. 
Set
$\Sigma:=\partial\A.$  

If $\A_{/N}$ is terminal, the proof is finished. So we may assume  
\begin{equation}
\label{eq:quotient-boundary-compatible-at-N}
\Sigma_{/N}
=
\partial(\A_{/N}).
\end{equation} Take any proper subspace $M$ such that $N\subset M\subset V$. The natural isomorphism \[(V/N)/(M/N)\simeq V/M\] identifies  
\begin{equation}
\label{eq:quotient-associativity-A}
(\A_{/N})_{/(M/N)}
=
\A_{/M},\qquad 
(\Sigma_{/N})_{/(M/N)}
=
\Sigma_{/M}.
\end{equation} Since $\A$ is Liouville admissible, 
\[\text{either }\A_{/M} \text{ is terminal,}\quad \text{or }\ \Sigma_{/M}=\partial (\A_{/M}).\]
In the first case, \((\A_{/N})_{/(M/N)}\simeq\A_{/M}\) is terminal. In the second case,\[
\bigl(\partial(\A_{/N})\bigr)_{/(M/N)}
=
(\Sigma_{/N})_{/(M/N)}
=
\Sigma_{/M}
=
\partial(\A_{/M})
=
\partial\bigl((\A_{/N})_{/(M/N)}\bigr).
\]
Thus, we have verified that  \(\A_{/N}\) is Liouville
admissible.
\end{proof}

We proceed to state a dichotomy result that is essential to understand the geometry of Liouville admissible sets.
\begin{proposition}
\label{prop:boundary-dichotomy-reduction}
Let \(\A\subset\mathcal Q(V)\) be admissible and set
\(\Sigma=\partial\A\). Let \(N\subset V\) be a proper nontrivial
\(\mathbb F\)-linear subspace. Then
\[
\partial(\A_{/N})
\subset
\Sigma_{/N}
\subset
\A_{/N}.
\]
Moreover, 
\begin{equation}%\label{goodcase}
\text{either }\ \partial(\A_{/N})
=
\Sigma_{/N},\quad \text{or}\ \ 
\Sigma_{/N}
=
\A_{/N}.\end{equation}
In the second case, there exists a real linear hyperplane
\(P\subset\mathcal Q(V)\) supporting \(\A\) such that
\[
q_N^*\mathcal Q(V/N)\subset P.
\]
\end{proposition}

\begin{proof}
Since \(\Sigma\subset\A\), one has
\(\Sigma_{/N}\subset\A_{/N}\).

Let \(B\in\partial(\A_{/N})\). Since \(\A_{/N}\) is
closed, \(q_N^*B\in\A\). If
\(q_N^*B\in\operatorname{int}\A\), continuity of \(q_N^*\) implies that
\(B\) lies in an open subset of \(\A_{/N}\), contradicting
\(B\in\partial(\A_{/N})\). Hence
\[
q_N^*B\in\A\setminus\operatorname{int}\A
=
\partial\A
=
\Sigma,
\]
so \(B\in\Sigma_{/N}\). Thus
\begin{equation}\label{new1}
\partial(\A_{/N})
\subset
\Sigma_{/N}.
\end{equation}

If two sets in \eqref{new1} are equal,
there is nothing further to prove. Otherwise,  choose
\[
B_0
\in
\Sigma_{/N}
\setminus
\partial(\A_{/N}).
\]
Since \(\Sigma_{/N}\subset\A_{/N}\) and \(\A_{/N}\) is closed,
\(B_0\in\operatorname{int}(\A_{/N})\).

Set \(X_0=q_N^*B_0\). Then \(X_0\in\Sigma=\partial\A\). Since \(\A\) is
closed and convex with nonempty interior, there exists a nonzero real
linear functional \(\ell\) on \(\mathcal Q(V)\) such that for all $X\in\A$,
\begin{equation}\label{add0}
\ell(X)\ge\ell(X_0).
\end{equation}

Let \(E\in\mathcal Q(V/N)\). Since \(B_0\) is an interior point of
\(\A_{/N}\), one has \(B_0+tE\in\A_{/N}\) for all sufficiently small
\(t\in\mathbb R\). Hence
\begin{equation}
X_0+tq_N^*E\in\A.
\end{equation}
Applying \eqref{add0} for both signs of small \(t\) gives
\[
\ell(q_N^*E)=0.
\]
Therefore
\[
q_N^*\mathcal Q(V/N)\subset P:=\ker\ell.
\]

Now let \(B\in\A_{/N}\). Then
\(q_N^*B\in\A\cap P\). Since \(P\) supports \(\A\),
\[
\A\cap P\subset\partial\A=\Sigma.
\]
Hence \(q_N^*B\in\Sigma\), so \(B\in\Sigma_{/N}\). Therefore
\(
\A_{/N}\subset\Sigma_{/N}.
\)
Since the reverse inclusion is obvious,
\(
\Sigma_{/N}=\A_{/N}\). 
\end{proof}
We proceed to establish the following proposition.

\begin{proposition}
\label{lem:liouville-admissible-stability}
Liouville admissibility is preserved under finite intersections and
under the natural \(GL_{\mathbb F}(V)\)-action on \(\mathcal Q(V)\).
\end{proposition}
\begin{proof}
Let \(\A_1,\dots,\A_m\subset\mathcal Q(V)\) be Liouville admissible and
set
\[
\A=\bigcap_{j=1}^m\A_j.
\]
By Lemma~\ref{lem:admissible-stability},
\(\A\) is admissible. Set $\Sigma=\partial\A$.

For every nonzero proper subspace \(N\subset V\), $\A_{/N}\subset(\A_j)_{/N}$.
If there exists $j$ such that $(\A_j)_{/N}$ is contained in a terminal set, then so is $\A_{/N}$. Otherwise, by Proposition~\ref{prop:boundary-dichotomy-reduction}, $\partial_{\mathcal{Q}(V/N)} ((\A_j)_{/N}) = (\Sigma_j)_{/N}$ for all $j$. In particular, $(\Sigma_j)_{/N}$ has empty interior. Since  
\[
\Sigma_{/N} \subset \bigcup_j (\Sigma_j)_{/N},
\]
$\Sigma_{/N}$  also has empty interior in $\mathcal{Q}(V/N)$.
Applying \autoref{prop:boundary-dichotomy-reduction} to $\A$, we have
\[\Sigma_{/N} = \partial (\A_{/N}).\] 
Therefore, \(\A\) is
Liouville admissible.

Now let \(\varphi\in GL_{\mathbb F}(V)\).  \(\varphi^*\A\) is admissible by Lemma~\ref{lem:admissible-stability}.
 For any nonzero subspace
\(N\subset V\), there exists a natural linear isomorphism
\(\bar\varphi:V/N\to V/\varphi(N)\) such that
\[
(\varphi^*\A)_{/N}=\bar\varphi^*(\A_{/\varphi(N)}),\qquad
(\varphi^*\Sigma)_{/N}=\bar\varphi^*(\Sigma_{/\varphi(N)}).
\]
Since \(\A\) is Liouville admissible, either \(\A_{/\varphi(N)}\) is contained in
a terminal set \(L\), or \(\Sigma_{/\varphi(N)}=\partial(\A_{/\varphi(N)})\).  In the
first case, \((\varphi^*\A)_{/N}\subset\bar\varphi^*(L)\) with
\(\bar\varphi^*(L)\) a terminal set.  In the second case, because
\(\bar\varphi^*\) is a linear homeomorphism,
\(\partial((\varphi^*\A)_{/N})=(\varphi^*\Sigma)_{/N}\).  Thus
\(\varphi^*\A\) satisfies the Liouville admissible dichotomy, and is
Liouville admissible.
\end{proof}

\begin{remark}
\label{rem:anisotropic-examples}
The preceding stability results provide many anisotropic examples of Liouville admissible sets.
\end{remark}
\subsection{Spectral Sets}
Although the definition of an admissible set does not require a metric
on \(V\), many applications come equipped with a distinguished Euclidean,
respectively, Hermitian, metric \(g\in\mathcal Q_{++}(V)\). Note that $g$ may or may not be the same as the metric $h$ in \ref{Cond:subharmonic} in Definition~\ref{def:admissible}. Once such a
metric is fixed, every \(A\in\mathcal Q(V)\) may be identified with the
self-adjoint endomorphism \(g^{-1}A\), and hence has a well-defined
unordered eigenvalue vector
\[
\lambda_g(A)=(\lambda_1,\ldots,\lambda_n).
\] Also, for any proper linear subspace $W$, we define 
$$W^{\perp_g}:=\{v\in V, g(v,w)=0, \ \forall w\in W\}, $$ which is linear isomorphic to $V/W$.
This leads to the following notion.

\begin{definition} 
\label{def:spectral-set}

Let \((V,g)\) be a finite-dimensional real or complex 
vector space with a metric \(g\in\mathcal Q_{++}(V)\). A set
\(\A\subset\mathcal Q(V)\) is called \emph{spectral} if
\[
U^*\A=\A
\]
for every $g$-invariant linear map \(U\), i.e. for every 
\(U\in O(V,g)\) in the real case and every
\(U\in U(V,g)\) in the complex case.

Equivalently, there exists a permutation-invariant set
\(
\Lambda\subset\mathbb R^n
\)
such that
\begin{equation}\label{Lambda}
\A = \{A\in\mathcal Q(V):\lambda_g(A)\in\Lambda\}.
\end{equation}

\end{definition}
 \begin{theorem}
\label{lem:spectral-automatic}

Let \((V,g)\) be given as in \autoref{def:spectral-set}, and
let \(\A\subset\mathcal Q(V)\) be admissible and spectral.  Then
\(\A\) is Liouville admissible.

\end{theorem}

\begin{proof}

Set \(\Sigma=\partial\A\), and let \(N\subset V\) be a proper nontrivial
\(\mathbb F\)-linear subspace. Write
\[
V=E\oplus N,\]
where
$E=N^{\perp_g}\simeq V/N$. By Proposition~\ref{prop:boundary-dichotomy-reduction}, either
\(\partial(\A_{/N})=\Sigma_{/N}\) holds
or else \(\Sigma_{/N}=\A_{/N}\) and there exists a supporting
hyperplane \(P=\ker\ell\subset\mathcal Q(V)\) satisfying
\(q_N^*\mathcal Q(V/N)\subset P\). The first case fits the definition of Liouville admissibility.

We assume the second.
Using the trace pairing, write \(\ell(X)=\Tr_g(SX)\) for some nonzero
\(S\in\mathcal Q(V)\). Since \(0\in\A\) and
\(\A+\mathcal Q_+(V)\subset\A\), we have
\(\mathcal Q_+(V)\subset\A\). Hence
\(\Tr_g(SX)\ge0\) for all \(X\in\mathcal Q_+(V)\), which implies
\(S\in\mathcal Q_+(V)\).

Since \(q_N^*\mathcal Q(V/N)\subset\ker\ell\), the \(E\)-block of
\(S\) vanishes. Because \(S\ge0\), all mixed \(E\)-\(N\) blocks also
vanish, and therefore \(S\) is supported entirely on the \(N\)-factor.
Choose \(j>m\), where \(m=\dim_{\mathbb F}(V/N)\), such that
\(S_{jj}>0\).

Let $\Lambda\subset \R^n$ satisfy \eqref{Lambda}. Thus, $\Lambda$ is permutation invariant and satisfies the elliptic condition
\[
\Lambda+\mathbb R_{\ge0}^n\subset\Lambda.
\]

The quotient reduction corresponds to
\[
\Lambda_{/N}
=
\{\xi\in\mathbb R^m:(\xi,0)\in\Lambda\}.
\]

Let \(\xi\in\Lambda_{/N}\). Then \((\xi,0)\in\Lambda\). Permuting the
\(i\)-th coordinate with the \(j\)-th coordinate and applying
\(\ell\ge0\) yields \(S_{jj}\xi_i\ge0\). Since \(S_{jj}>0\), we obtain
\(\xi_i\ge0\) for every \(i\). Thus
\(
\Lambda_{/N}\subset\mathbb R_{\ge0}^m,
\)
and consequently
\[
\A_{/N}\subset\mathcal Q_+(V/N).
\]

By Proposition~\ref{prop:positive-cone-terminal},
\(\mathcal Q_+(V/N)\) is terminal. Therefore, for every proper
nontrivial subspace \(N\subset V\), either \(\Sigma_{/N}=\partial(\A_{/N})\)
or \(\A_{/N}\) is contained in a terminal set. Hence \(\A\) is
Liouville admissible.
\end{proof}

Theorem~\ref{lem:spectral-automatic} shows that the Liouville admissibility of a spectral admissible set is automatic once admissibility has been verified.

\section{From Univariate G\aa rding To Liouville Admissible Sets}
\label{sec:applications-univariate-garding}
In this section, we construct a large class of spectral Hessian equations from univariate G{\aa}rding polynomials.

\subsection{Univariate G{\aa}rding Polynomials And Polarization}

\begin{definition}
Let \(f\in\mathbb R[t]\) be a polynomial of degree \(d\). If \(f\)
has at least one real root, we denote by \(r(f)\) the largest real
root of \(f\). We say that \(f\) satisfies the \emph{monotone root sequence}
condition, or MRS condition, if every derivative \(f^{(i)}\),
\(0\le i\le d-1\), has a real root and
\[
r(f)\ge r(f')\ge\cdots\ge r(f^{(d-1)}).
\]
A univariate polynomial satisfying this condition is called a \emph{univariate
G{\aa}rding polynomial}.
\end{definition}

The above definition was first introduced by Lin in \autocite{MR4604837}, where a univariate polynomial $f(t)$ that satisfies MRS is called \emph{right-Noetherian}. 

\begin{remark}
Every real stable univariate polynomial is univariate G{\aa}rding. In
fact, real stability in one variable is equivalent to having only real
roots. Rolle's theorem then implies that all derivatives are
real-rooted and that their largest roots satisfy the MRS condition.
\end{remark}

For \(f(t)=\sum_{k=0}^d a_k t^k\) and 
\(n\ge d\), we define the polarization of $f$ as
\begin{equation}
    \label{pol}
\Pol_n(f)(\mathbf x)=\Pol(f)(\x)
=
\sum_{k=0}^d
a_k\frac{\sigma_k(\mathbf x)}{\binom nk}.
\end{equation}

\begin{theorem}[Univariate G{\aa}rding polarization]
\label{thm:univariate-garding-polarization}
Let \(f(t)\) be a univariate G{\aa}rding polynomial of degree \(d\), and 
\(n\ge d\), and set $g(\x)=
\Pol_n(f)(\mathbf x)$. Then
\begin{enumerate}
\item \(g\) is a G{\aa}rding polynomial on \(\mathbb R^n\);

\item $\{g>0\}$ has a unique connected component, \(\mathcal C_g\) passing the positive ray test (PRT), i.e.,
\[
\mathcal C_g+\mathbb R_+^n\subset\mathcal C_g;
\]

\item  \(\mathcal C_g\) is convex;

\item for every multi-index \(\alpha\in\mathbb N^n\) with
\(\partial^\alpha g\not\equiv0\), the derivative
\(\partial^\alpha g\) is G{\aa}rding and possesses a G{\aa}rding
component \(\mathcal C_{\partial^\alpha g}\) satisfying
\[
\mathcal C_g\subset\mathcal C_{\partial^\alpha g}.
\]
\end{enumerate}
\end{theorem}

\begin{proof}
See \autocite{fang2026gaardingpolynomials, fang2026idealgaardingpolynomials}.
See also \autocite{MR4604837} for the convexity argument of the $d=n$ case.
\end{proof}

We now give the detailed form of \autoref{thm:intro-example}, which was
stated in the introduction.

\begin{theorem}
\label{thm:polarized-garding-liouville}
Let \(\mathbb F\in\{\mathbb R,\mathbb C\}\), and let \((V,h)\) be \(\mathbb F^n\) with its standard Euclidean (respectively, Hermitian) metric. Let \(f\in\mathbb R[t]\) be a univariate G{\aa}rding
polynomial of degree \(d\le n\), normalized by \(r(f)=0\). Set
\(g=\Pol_n(f)\) as in (\ref{pol}), let \(\mathcal C_g\) be the G{\aa}rding
component of $g$ from
Theorem~\ref{thm:univariate-garding-polarization}, and define the associated spectral set
\[
\A_g
=
\{A\in\mathcal Q(V):
\lambda_h(A)\in\overline{\mathcal C_g}\},
\qquad
\Sigma_g=\partial\A_g.
\]
Then \(\A_g\) is 
Liouville admissible.
\end{theorem}

\begin{proof}
By Theorem~\ref{thm:univariate-garding-polarization}, the polynomial
\(g\) is G{\aa}rding, the component \(\mathcal C_g\) is convex and satisfies PRT.
Since \(r(f)=0\), one has \(0\in\overline{\mathcal C_g}\), and hence
\(\overline{\mathcal C_g}+\mathbb R_{\ge0}^n
\subset\overline{\mathcal C_g}\).

Because \(g\) is symmetric, \(\overline{\mathcal C_g}\) is permutation
invariant. Therefore, its spectral lift
\[
\A_g
=
\{A\in\mathcal Q(V):
\lambda_h(A)\in\overline{\mathcal C_g}\}
\]
is closed, convex by the Davis convexity theorem~\cite{MR90572}, and contains \(0\). If \(A\in\A_g\) and \(P\in\mathcal Q_+(V)\), then eigenvalue monotonicity gives
\(\lambda_h(A+P)\in
\overline{\mathcal C_g}+\mathbb R_{\ge0}^n\subset\overline{\mathcal C_g}\), hence \(A+P\in\A_g\).
Thus \(\A_g+\mathcal Q_+(V)\subset\A_g\).

We next verify the subharmonic half-space condition. We claim
\begin{equation}\label{spec-trace}
\overline{\mathcal C_g}
\subset
\left\{
x\in\mathbb R^n:
\sum_{i=1}^n x_i\ge0
\right\}.
\end{equation}

Since \(\overline{\mathcal C_g}\) is closed, convex, proper, and
\(0\in\partial\overline{\mathcal C_g}\), consider a supporting hyperplane at $0$. There exists a nonzero vector
\(\mathbf a=(a_1,\dots,a_n)\) such that \(\mathbf a\cdot \x\ge0\) on
\(\overline{\mathcal C_g}\). By PRT, $a_i\ge0$. Hence
\(\sum_i a_i>0\).

Since \(\overline{\mathcal C_g}\) is permutation invariant, every
\(\sigma a\), \(\sigma\in S_n\), is also a supporting normal.
After summing and averaging, $\mathbf 1$ is  a supporting normal, proving \eqref{spec-trace}, which indicates that
 \(\A_g\) satisfies the subharmonic half-space condition in
Definition~\ref{def:admissible}. 

We have shown that $\A_g$ is admissible. Finally, Theorem~\ref{lem:spectral-automatic} implies
that \(\A_g\) is Liouville admissible.
\end{proof}

\begin{remark}
\label{rem:explicit-Cg-Ag}
Let $f$ and $g$ be given as above.
For \(S\subset\{1,\dots,n\}\), write
\[
\partial_S:=\prod_{i\in S}\partial_i,
\qquad
\lambda|S:=(\lambda_i)_{i\notin S},
\qquad
|S|=s.
\]
Since \(g\) is multi-affine,
\[
\partial_S g(\lambda)
=
\sum_{k=s}^d
a_k
\frac{\sigma_{k-s}(\lambda|S)}{\binom nk}.
\]
The distinguished G{\aa}rding component \(\mathcal C_g\) is the following
\[
\mathcal C_g
=
\left\{
\lambda\in\mathbb R^n:
\sum_{k=s}^d
a_k
\frac{\sigma_{k-s}(\lambda|S)}{\binom nk}
>0
\text{ for all }
S
\text{ with }
\partial_S g\not\equiv0
\right\}.
\]
The associated spectral admissible set is
\[
\A_g
=
\{A\in\mathcal Q(\mathbb F^n):
\lambda(A)\in\overline{\mathcal C_g}\}.
\]
Thus, the boundary equation \(\Hess_{\mathbb F}u\in\Sigma_g=\partial \A_g\) is the
polynomial equation \(g(\lambda(\Hess_{\mathbb F}u))=0\), understood on
the admissible set \(\A_g\). 
\end{remark}

\begin{example}
\label{ex:k-hessian-branch}
Let
\[
f(t)=t^k,
\qquad
1\le k\le n.
\]
Then \(f\) is a univariate G{\aa}rding polynomial with \(r(f)=0\), and
\[
g=\Pol_n(f)=\frac{\sigma_k(\x)}{\binom nk}.
\]
We identify the associated Liouville admissible set with the usual
closed \(k\)-Hessian cone.

The G{\aa}rding component of \(\sigma_k\) is
\[
\Gamma^+_{k}
=
\{\lambda\in\mathbb R^n:
\sigma_1(\lambda)>0,\dots,\sigma_k(\lambda)>0\},
\]
namely, the component of \(\{\sigma_k>0\}\) that contains the positive
orthant.

Consequently,
\[
\A_k
:=
\{A\in\mathcal Q(V):
\lambda(A)\in\overline{\Gamma^+_{k}}\}
=
\{A\in\mathcal Q(V):
\sigma_1(A)\ge0,\dots,\sigma_k(A)\ge0\}.
\]
Thus \(\A_k\) is exactly the usual closed \(k\)-Hessian cone,
which we may write as
\[
\A_k=\overline{\Gamma^+_{k}}(\mathcal Q(V)).
\]

Its boundary equation is
\[
\Sigma_k:=\partial\A_k
=
\{A\in\A_k:\sigma_k(A)=0, \sigma_j(A)\geq 0, \ j<k\}.
\]

By \autoref{thm:polarized-garding-liouville}, \(\A_k\) is Liouville
admissible. Therefore, every bounded and  uniformly $\alpha$-H\"{o}lder continuous ($0 <\alpha \leq 1$) viscosity solution of
\[
\Hess_{\mathbb F}u\in\Sigma_k
\] is constant. For the complex case, with a stronger $C^1$ bound, this was first proved by \citeauthor{MR3636634} \autocite{MR3636634}. Also see \autocite{MR3807322}.
\end{example}

\begin{example}
Let
\[
f(t)=t^n+n t^{n-1}=t^{n-1}(t+n).
\]
Then $f$ is real stable and G\aa rding. Note that
\[
\Pol_n(f)=\sigma_n+\sigma_{n-1}.
\]
Hence, the associated spectral set $\A$ is Liouville admissible, which by direct computation equals 
$$\A=\{A\in\Q(V):\ \sigma_k(A)\geq0,\ \forall k\leq n-1;\ \sigma_{n-1}(A)+\sigma_n(A)\geq 0\ \}.$$ Then every bounded, uniformly $\alpha$-H\"{o}lder continuous viscosity solution of
\[
\sigma_n(\Hess_{\mathbb F}u)
+
\sigma_{n-1}(\Hess_{\mathbb F}u)
=
0,\quad \Hess_{\mathbb F} u\in\A
\]
is constant. 
The complex Liouville theorem for \(\sigma_n+\sigma_{n-1}=0\) was first
obtained by \citeauthor{fu2026criticallyzequationkahler} \autocite{fu2026criticallyzequationkahler}
as a key blow-up input in the study of the critical LYZ equation.
Our framework gives an alternative proof, recovering their result as
a special case of the general Liouville admissibility condition.

\end{example}

\begin{example}We give a non-stable example.
Let
\[
f(t)=t^3+3t^2+3t=(t+1)^3-1.
\]
Since
\(
r(f)=0,\ 
r(f')=r(f'')=-1,
\)
the polynomial \(f\) satisfies the monotone root sequence condition. Hence, $f$ is G\aa rding but not real stable.
For \(n\ge3\),
\[
\Pol_n(f)
=
\frac{\sigma_3}{\binom n3}
+
3\frac{\sigma_2}{\binom n2}
+
3\frac{\sigma_1}{n}.
\]
The corresponding  Hessian equation is 
\[
\sigma_3(\Hess_{\mathbb F}u)
+
(n-2)\sigma_2(\Hess_{\mathbb F}u)
+
\frac{(n-1)(n-2)}2
\sigma_1(\Hess_{\mathbb F}u)
=
0,
\]
and the associated admissible branch is Liouville admissible. For \(n=3\),
the equation may be written as
\[
\sigma_3(\Hess_{\mathbb F}u+I)=1.
\]
\end{example}

\begin{remark}
The ideal G{\aa}rding polynomial framework suggests additional
Liouville admissible branches beyond the polarized univariate examples,
in the spirit of the constructions in \cite{MR4781476}. We leave
these extensions for future work.
\end{remark}

\section{Technical Tools} 
In the next three sections, we prove \autoref{thm:main}. In this section,
we gather some technical tools to be used in the proof.  The analytic tools are stated in the real setting. They apply in the complex case after identifying \(\mathbb C^n\) with \(\mathbb R^{2n}\);
the complex Hessian is obtained from the real Hessian by a fixed linear
projection.

\subsection{Mollification}

Let \(\rho \colon \mathbb R^n \to \mathbb R\) be a fixed smooth, radially nonincreasing, nonnegative function with compact support in the unit ball
\(\mathbb B_1 \subset \mathbb R^n\),  satisfying
\[
  \inf_{B_{1/2}(0)}\rho>0,
\]
and normalized so that
\[
  \int_{\mathbb R^n} \rho(y)\,dy = 1 .
\]
For a continuous function \(w \colon \mathbb R^n \to \mathbb R\), we define its
mollification \([w]_\varepsilon\) at scale \(\varepsilon>0\) as
\begin{equation}
  \label{eq:moll-orig}
  [w]_\varepsilon(z)
  =
  \int_{\mathbb R^n} w(z-\varepsilon y)\rho(y)\,dy .
\end{equation}

\subsection{Cartan's lemma and its density consequence}

The large-scale behaviour of bounded subharmonic functions is controlled
by the following lemma of Cartan.

\begin{theorem}\autocites[][Theorem 3.1]{MR3636634}[Lemma 3.10]{MR460672}
  \label{thm:Cartan}
  Let \(w\) be a bounded subharmonic function on \(\mathbb R^n\),
  \(n\ge2\), and fix any \(q>n-2\).  If
  \(a=\sup_{\mathbb R^n}w\), then
  \[
    \lim_{\|x\|\to\infty,\,
    x\in\mathbb R^n\setminus A}
    w(x)=a,
  \]
  where the exceptional set \(A\) is contained in at most a countable
  union of balls \(B(x_k,r_k)\) such that
  \[
    \sum_{k=1}^\infty
    \left(\frac{r_k}{\|x_k\|}\right)^q
    <\infty .
  \]
\end{theorem}

We only need the following consequence.

\begin{corollary}
  \label{cor:Cartan-density}\autocites[]{MR3636634}
  Let \(u\) be a bounded non-negative subharmonic function on
  \(\mathbb R^n\), \(n\ge2\), with
  \(\sup_{\mathbb R^n}u=1\).  Then for every \(z\in\mathbb R^n\),
  \[
    \lim_{r\to\infty} [u^2]_r(z)
    =
    \lim_{r\to\infty} [u]_r(z)
    =
    1 .
  \]
\end{corollary}

\begin{proof}
  We apply Theorem~\ref{thm:Cartan} with \(q=n\) and let \(A\) be the exceptional
  set provided by the lemma.  It suffices to prove that the asymptotic
  density of \(A\) vanishes:
  \[
    \lim_{r \to \infty}
    \frac{|B(0,r) \cap A|}{|B(0,r)|}
    =
    0 .
  \]
  The rest of the proof is a standard packing estimate.  Fix
  \(0<\varepsilon<1\) and choose \(K\) such that
  \[
    \sum_{k=K}^\infty
    \left(\frac{r_k}{\|x_k\|}\right)^n
    <
    \varepsilon .
  \]
  Then for \(k\ge K\),
  \[
    r_k \le \varepsilon^{1/n}\|x_k\|,
  \]
  and hence
  \[
    \|x_k\|-r_k
    \ge
    \bigl(1-\varepsilon^{1/n}\bigr)\|x_k\|.
  \]
  Therefore
  \[
    \begin{aligned}
    \limsup_{r\to\infty}
    \frac{|B(0,r) \cap A|}{|B(0,r)|}
    &\le
    \limsup_{r\to\infty}
    \frac{1}{r^{n}}
    \sum_{k\,:\,\|x_k\|-r_k \le r} r_k^{n}  \\
    &\le
    \limsup_{r\to\infty}
    \sum_{k\le K} \frac{r_k^{n}}{r^{n}}
    +
    \sum_{k\ge K,\,\|x_k\|-r_k \le r}
    \frac{r_k^{n}}{(\|x_k\|-r_k)^{n}}   \\
    &\le
    \limsup_{r\to\infty}
    \sum_{k\le K} \frac{r_k^{n}}{r^{n}}
    +
    \bigl(1-\varepsilon^{1/n}\bigr)^{-n}
    \sum_{k=K}^\infty
    \left(\frac{r_k}{\|x_k\|}\right)^n  \\
    &\le
    \bigl(1-\varepsilon^{1/n}\bigr)^{-n}\varepsilon .
    \end{aligned}
  \]
  Since \(0<\varepsilon<1\) is arbitrary, the density is zero.  This
  proves the asserted convergence of the mollifications.
\end{proof}
\subsection{Upper envelopes and convex envelopes}

To prove that mollifications of viscosity subsolutions remain
subsolutions, we use upper \(\varepsilon\)-envelopes together with a
convex-envelope estimate.  The following definition and theorem are
taken from~\cite{MR1351007}.

\begin{definition}[Upper \(\varepsilon\)-envelope]
  \autocite[Section 5.1]{MR1351007}
  Let \(u\) be a continuous function in a domain
  \(\Omega\subset \mathbb R^n\), and let \(H\) be an open set with
  \(\overline H \subset \Omega\).  For \(\varepsilon>0\), the
  \emph{upper \(\varepsilon\)-envelope} of \(u\) (with respect to \(H\))
  is
  \[
    u^\varepsilon(z_0)
    =
    \sup_{z\in\overline H}
    \left\{
      u(z)
      -
      \frac{1}{\varepsilon}|z-z_0|^2
    \right\}
    +
    \varepsilon,
    \qquad z_0\in H .
  \]
\end{definition}

The next theorem was originally stated in
\cite{MR1351007} for real uniformly elliptic equations. 
The proof uses only the translation invariance of the equation and
therefore carries over verbatim to our subequation
\(\Hess_{\mathbb F}w\in\A\).
\begin{theorem}
  \label{thm:upper-envelope}
  \autocite[Theorem 5.1]{MR1351007}
  Let $\mathbb F \in \{\R, \C\}$. Let $V$ denote a finite-dimensional vector space over $\mathbb F$, let \(\A\) be admissible, and let \(w\) be a continuous viscosity subsolution of
  \(\Hess_{\mathbb{F}}(w)\in\A\) in a domain \(\Omega\subset V\).  Let
  \(H\Subset\Omega\).  Then for all sufficiently small
  \(\varepsilon>0\), the upper \(\varepsilon\)-envelope
  \(w^\varepsilon\) satisfies:
  \begin{enumerate}[label=(\roman*)]
    \item \(w^\varepsilon\in C(H)\) and
          \(w^\varepsilon \downarrow w\) uniformly on \(H\) as
          \(\varepsilon\to0\);
          \label{thm:upper-envelope:item1}

    \item at every \(z_0\in H\), there exists a concave paraboloid of
          opening \(2/\varepsilon\) touching \(w^\varepsilon\) from below
          at \(z_0\); in particular, \(w^\varepsilon\) is punctually
          second-order differentiable almost everywhere in \(H\);
          \label{thm:upper-envelope:item2}

    \item if \(H_1\) is an open set with \(\overline H_1\subset H\), then
          for all sufficiently small \(\varepsilon>0\),
          \(w^\varepsilon\) is a viscosity subsolution of
          \(\Hess_{\mathbb{F}}(w^\varepsilon)\in\A\) in \(H_1\); consequently
          \(\Hess_{\mathbb{F}}(w^\varepsilon)(z)\in\A\) for almost every
          \(z\in H_1\).
          \label{thm:upper-envelope:item3}
  \end{enumerate}
\end{theorem}

\begin{theorem}
  \label{thm:caffarelli}
  \autocite[Lemma 3.5]{MR1351007}
  Let \(u\) be a continuous function in
  \(\overline B_d\subset\mathbb R^n\) with \(u\ge0\) on
  \(\partial B_d\).  Extend \(u\) by zero outside \(B_d\), set
  \(u^-=\max\{-u,0\}\), and let \(\Gamma_u\) be the convex envelope
  of \(-u^-\) in \(B_{2d}\).  Assume that for every
  \(x_0\in\overline B_d\cap\{u=\Gamma_u\}\) there exists a convex
  paraboloid of opening \(K\) touching \(\Gamma_u\) from above at
  \(x_0\) in \(B_\varepsilon(x_0)\), where \(K>0\) and
  \(0<\varepsilon\le d\).  Then \(\Gamma_u\in C^{1,1}(\overline B_d)\),
  and there exists a set \(A\subset B_d\) with
  \(|B_d\setminus A|=0\) such that \(\Gamma_u\) is second-order
  differentiable at every \(x\in A\).  Moreover,
  \[
    \sup_{B_d} u^-
    \le
    C(n)\, d
    \left(
      \int_{A\cap\{u=\Gamma_u\}}
      \det D^2\Gamma_u
    \right)^{1/n},
  \]
  where \(C(n)\) depends only on \(n\).
\end{theorem}

\section{Properties of viscosity solutions}

In this section, we collect the basic properties of viscosity solutions of the inclusion~\eqref{equation} under the admissibility conditions of
Definition~\ref{def:admissible}.  

\begin{proposition}\label{prop:Cond-properties}
  Let $\A\subset\mathcal{Q}(V)$ be admissible. Let \(g\in\mathcal Q_{++}(V)\) be a metric such that $\Tr_g A\ge 0$ for all $A\in\A$. Set $\Sigma=\partial\A$.
  \begin{enumerate}[label=(P\arabic*)]
    \item \label{PropConsCond:P1}Every \(C^2\) solution of \(\Hess_{\mathbb F}w\in\Sigma\) is a
viscosity solution.

    \item \label{PropConsCond:P2}
     If \(w_\ell\to w\) locally uniformly and each \(w_\ell\) is a viscosity
subsolution, supersolution, or solution, then \(w\) has the same
property.
    \item \label{PropConsCond:P3}
      If $w\in C^2$ satisfies $0\le w\le 1$ and $\Hess_{\mathbb F} w\in\A$, then
      $\Hess_{\mathbb F}\bigl(\frac12 w^2\bigr)\in\A$.
    \item \label{PropConsCond:P4}
      If $v,w\in C^2$ satisfy $\Hess_{\mathbb F} v,\Hess_{\mathbb F} w\in\A$, then $\Hess_{\mathbb F}\bigl(\frac12(v+w)\bigr)\in\A$.
    \item \label{PropConsCond:P5}
      Let $v,w$ be continuous viscosity subsolutions.  Then
      $\frac12(v+w)$ is also a viscosity subsolution.
      Consequently, for any continuous viscosity subsolution $w$, its
      standard smooth mollification $[w]_\varepsilon$ (defined by
      \eqref{eq:moll-orig} with respect to metric $g$) is again a (smooth) viscosity subsolution.
      Moreover, because $w$ is subharmonic, $[w]_\varepsilon \ge w$.
    \item \label{PropConsCond:P6}
      Let $\Omega\subset V$ be a bounded domain, $u$ a continuous
      viscosity supersolution and $v\in C^2$ a viscosity subsolution in
      $\Omega$.  If $\liminf_{z\to\partial\Omega}(u-v)(z)\ge0$, then
      $u\ge v$ in $\Omega$.
    \item \label{PropConsCond:P7}
      Let $U\in O(V)$ in the real case (resp.\ $U\in\mathrm{U}(V)$ in
      the complex case). If $w$ satisfies $\Hess_{\mathbb F} w\in\Sigma$ in the
      viscosity sense, then $v(z)=w(Uz)$ satisfies
      $\Hess_{\mathbb F} v\in U^\ast\Sigma U$ in the viscosity sense, where, for \(S\subset\mathcal Q(V)\), we write
\[
U^\ast S U:=\{U^\ast A U:A\in S\}.
\]

    \item \label{PropConsCond:P8}
      Let $w_\ell \to w$ locally uniformly, and let $U_\ell\in O(V)$ (resp.\ $U_\ell\in\mathrm{U}(V)$) be
      orthogonal (resp.\ unitary) matrices converging to $U$.  If
      $\Hess_{\mathbb F} w_\ell\in U_\ell^\ast\Sigma U_\ell$ in the viscosity sense, then
      $\Hess_{\mathbb F} w\in U^\ast\Sigma U$ in the viscosity sense.
    \item \label{PropConsCond:P9}
           Identify \(V\) with \(\mathbb F^n\) so that the metric
      \(g\) from Definition~\ref{def:admissible} is the canonical
      Euclidean (if \(\mathbb F=\mathbb R\)) or Hermitian (if
      \(\mathbb F=\mathbb C\)) inner product.  Let
      \[
        N = \{0\}^{n-1}\times\mathbb F \subset V
      \]
      and write points of \(V\) as \(z=(z',\zeta)\) with
      \(z'\in\mathbb F^{n-1}\).  If \(w\) is a viscosity solution of
      \(\Hess_{\mathbb F}w\in\Sigma\) and \(w(z)=w(z',0)\) for all
      \(\zeta\), then the restricted function
      \[
        v(z') = w(z',0)
      \]
      is a viscosity subsolution of
      \(\Hess_{\mathbb F}v\in\A_{/N}\).  If, in addition,
      \(\Sigma_{/N}=\partial(\A_{/N})\), then \(v\) is also a viscosity
      supersolution of \(\Hess_{\mathbb F}v\in\Sigma_{/N}\), hence a full
      viscosity solution.

  \end{enumerate}
\end{proposition}

\begin{proof}
  \textbf{\ref{PropConsCond:P1}.}
  Let $\varphi$ be a $C^2$ function touching $w$ from above at $z_0$.
  Then $\Hess_{\mathbb F}\varphi(z_0)\ge\Hess_{\mathbb F} w(z_0)$.  Since
  $\Hess_{\mathbb F} w(z_0)\in\Sigma\subset\A$, \ref{Cond:up}
  implies that $\Hess_{\mathbb F}\varphi(z_0)\in\A$; hence $w$ is a viscosity subsolution.

  Now, let $\varphi$ touch $w$ from below at $z_0$.  Then
  $\Hess_{\mathbb F}\varphi(z_0)\le\Hess_{\mathbb F} w(z_0)$.  If
  $\Hess_{\mathbb F}\varphi(z_0)\in\operatorname{int}\A$, then
  \[
    \Hess_{\mathbb F} w(z_0)= \Hess_{\mathbb F}\varphi(z_0)
    +(\Hess_{\mathbb F} w(z_0)-\Hess_{\mathbb F}\varphi(z_0))
    \in \operatorname{int}\A + \mathcal Q_+(V)
    \subset \operatorname{int}\A,
  \]
  contradicting $\Hess_{\mathbb F} w(z_0)\in\Sigma=\partial\A$.  Therefore
  $\Hess_{\mathbb F}\varphi(z_0)\notin\operatorname{int}\A$, which means $w$ is a
  viscosity supersolution.  Thus $w$ is a viscosity solution.

  \textbf{\ref{PropConsCond:P2}.}
  We treat the subsolution case; the supersolution case is similar.  Let $\varphi\in C^2$ touch $w$ from above at $z_0$.  For
  any small ball $B_r(z_0)$, any $\varepsilon>0$ and any large index
  $\ell_0$, a suitable vertical translation of
  $\varphi+\varepsilon|z-z_0|^2$ touches $w_\ell$ from above at a point
  $z_\ell\in B_r(z_0)$ for all $\ell\ge \ell_0$.  Hence
  $\Hess_{\mathbb F}(\varphi+\varepsilon|z-z_0|^2)(z_\ell)\in\A$.  Letting $r\to0$
  and then $\varepsilon\to0$, the closedness of $\A$ yields
  $\Hess_{\mathbb F}\varphi(z_0)\in\A$.

  \textbf{\ref{PropConsCond:P3}.}
  In the real case,
  \[
    \Hess_{\mathbb R}(\tfrac12 w^2)
    = w\,\Hess_{\mathbb R} w
      + (\partial_i w\,\partial_j w)_{1\le i,j\le n},
  \]
  while in the complex case
  \[
    \Hess_{\mathbb C}(\tfrac12 w^2)
    = w\,\Hess_{\mathbb C} w
      + (\partial_i w\,\partial_{\bar\jmath} w)_{1\le i,j\le n}.
  \]
  In both cases, the second term on the right-hand side of the equality lies in $\mathcal Q_+(V)$.  If $w\in[0,1]$ and
  $\Hess_{\mathbb F} w\in\A$, \ref{Cond:origin} and \ref{Cond:convex} imply that 
  $w\,\Hess_{\mathbb F} w\in\A$, and \ref{Cond:up}
  then yields $\Hess_{\mathbb F} (\tfrac12 w^2)\in\A$.

  \textbf{\ref{PropConsCond:P4}.}
  This follows immediately from convexity of $\A$.

  \textbf{\ref{PropConsCond:P5}.}
  Fix any $\Omega \Subset \mathbb F^n$.
  Let $H\Subset\Omega$ and let $v^\varepsilon,w^\varepsilon$ be the upper
  $\varepsilon$-envelopes of $v,w$ in $H$
  (\autoref{thm:upper-envelope}).  For all sufficiently small
  $\varepsilon$, $v^\varepsilon$ and $w^\varepsilon$ are continuous, punctually
  second-order differentiable almost everywhere, and satisfy
  $\Hess_{\mathbb F} v^\varepsilon,\Hess_{\mathbb F} w^\varepsilon\in\A$ a.e.\ in a fixed
  open set $H_1\Subset H$.

  Suppose $\varphi\in C^2$ touches $\frac12(v^\varepsilon+w^\varepsilon)$
  from above at $z_0\in H_1$.  Choose a small ball
  $B_r(z_0)\subset H_1$ such that
  $\varphi\ge\frac12(v^\varepsilon+w^\varepsilon)$ on
  $\overline B_r(z_0)$ with equality at $z_0$.  For $\delta>0$, set
  \[
    u(z)=\varphi(z)+\delta|z-z_0|^2
         -\delta r^2-\frac12\bigl(v^\varepsilon(z)+w^\varepsilon(z)\bigr),
    \qquad z\in\overline B_r(z_0).
  \]
  Then $u(z_0)=-\delta r^2$ and $u\ge0$ on $\partial B_r(z_0)$.

  By \autoref{thm:upper-envelope}\ref{thm:upper-envelope:item2},
\(v^\varepsilon\) and \(w^\varepsilon\) are semiconvex.  Hence,
\[
  u=\varphi+\delta|z-z_0|^2-\delta r^2
    -\frac12(v^\varepsilon+w^\varepsilon)
\]
is semiconcave on \(\overline B_r(z_0)\).  Consequently, at every contact
point of \(\Gamma_u\) with \(u\), the function \(\Gamma_u\) can be touched
from above by a convex paraboloid with a uniform opening.  Thus the condition in \autoref{thm:caffarelli} is satisfied.
Applying \autoref{thm:caffarelli} to $u$ on $B_r(z_0)$ and
  using $u(z_0)<0$, we conclude that the contact set
  $\{u=\Gamma_u\}$ has positive Lebesgue measure.

  The functions $v^\varepsilon,w^\varepsilon$ are punctually second-order differentiable almost everywhere and satisfy
  $\Hess_{\mathbb F} v^\varepsilon,\Hess_{\mathbb F} w^\varepsilon\in\A$ at points where they are punctually second-order differentiable. Let $A$ be the set from \autoref{thm:caffarelli}. We can pick a point $z_1\in B_r(z_0)\cap\{u=\Gamma_u\}\cap A$ at which $v^\varepsilon,w^\varepsilon,u, \Gamma_u$ are all punctually twice differentiable,
  $\Hess_{\mathbb F} v^\varepsilon(z_1),\Hess_{\mathbb F} w^\varepsilon(z_1)\in\A$. At this point,
  \[
    \HessR u(z_1) \ge \HessR \Gamma_u(z_1) \ge 0 .
  \]
  The inequality $\HessR u(z_1)\ge0$ implies $\Hess_{\mathbb F} u(z_1)\ge0$.  From
  $\Hess_{\mathbb F} u=\Hess_{\mathbb F}(\varphi+\delta|z-z_0|^2)
  -\frac12(\Hess_{\mathbb F} v^\varepsilon+\Hess_{\mathbb F} w^\varepsilon)$ we obtain
  \begin{equation}
      \label{add4}
    \Hess_{\mathbb F}(\varphi+\delta|z-z_0|^2)(z_1)
    \ge \frac12\bigl(\Hess_{\mathbb F} v^\varepsilon(z_1)
                   +\Hess_{\mathbb F} w^\varepsilon(z_1)\bigr).
  \end{equation}
  The right-hand side of \eqref{add4} lies in $\A$ by \ref{Cond:convex}. \ref{Cond:up} then implies that $\Hess_{\mathbb F}(\varphi+\delta|z-z_0|^2)(z_1)\in\A$.  Letting $\delta\to0$
  and then $r\to0$, closedness of $\A$ yields
  $\Hess_{\mathbb F}\varphi(z_0)\in\A$.  Hence
  $\frac12(v^\varepsilon+w^\varepsilon)$ is a viscosity subsolution.  By
 \ref{PropConsCond:P2}, letting $\varepsilon\to0$ shows that $\frac12(v+w)$ is a viscosity subsolution.

  The mollification statement follows because $[w]_\varepsilon$ can be
  written as a limit of convex combinations of translates of $w$, each of
  which is a subsolution. By \ref{PropConsCond:P2}, the
  limit remains a subsolution.  The pointwise inequality
  $[w]_\varepsilon\ge w$ holds because $w$ is subharmonic.

  \textbf{\ref{PropConsCond:P6}.}
  Assume $\liminf_{z\to\partial\Omega}(u-v)(z)\ge0$. Fix $\eta>0$. By continuity,
  for all sufficiently small $\delta>0$,
  \[
    u+\eta > v+\delta|z|^2
    \qquad\text{on a neighbourhood of }\partial\Omega.
  \]

  Suppose that $v+\delta|z|^2 > u+\eta$ somewhere in $\Omega$.  Translate $v+\delta|z|^2$ downwards until it touches $u+\eta$ from below at some
  interior point $z_0$.  At $z_0$ we have
  \[
    \Hess_{\mathbb F}(v+\delta|z|^2)(z_0)\in\A+\mathcal Q_{++}(V) \subset \operatorname{int}\A,
  \]
  contradicting the definition of a viscosity supersolution $u+\eta$, which
  requires that the Hessian of any test function touching from below lies
  in $\overline{\mathcal Q(V)\setminus\A}$.  Hence
  $ u+\eta \geq v+\delta|z|^2 $ on $\Omega$.  Letting $\delta\to0$ and then $\eta\to0$ gives $v\le u$.

  \textbf{\ref{PropConsCond:P7}.}
Let $v(z)=w(Uz)$ with $U$ orthogonal (resp.\ unitary).
Suppose $\varphi\in C^2$ touches $v$ from above at $z_0$.
Set $\tilde\varphi(z)=\varphi(U^\ast z)$.  Then $v\le\varphi$
in a neighborhood of $z_0$ with equality at $z_0$.  Hence
$w\le\tilde\varphi$ in a neighborhood of $Uz_0$, with equality at $Uz_0$. By the definition of viscosity subsolution,
$\Hess_{\mathbb F}\tilde\varphi(Uz_0)\in\A$. Note that
\[
\Hess_{\mathbb F}\tilde\varphi(Uz_0)
= U\,\Hess_{\mathbb F}\varphi(z_0)\,U^\ast,
\]
which implies $\Hess_{\mathbb F}\varphi(z_0)\in U^\ast\A U$.
Since $\partial(U^\ast\A U)=U^\ast\Sigma U$, $v$ is a viscosity
subsolution of $\Hess_{\mathbb F} v\in U^\ast\Sigma U$.

Analogously, $v$ is a viscosity supersolution of
$\Hess_{\mathbb F} v\in U^\ast\Sigma U$.  Together, $v$ is a viscosity
solution of $\Hess_{\mathbb F} v\in U^\ast\Sigma U$.

  \textbf{\ref{PropConsCond:P8}.}
  Let $\varphi$ touch $w$ from above at $z_0$.  For $\delta>0$, a vertical translation of
  $\varphi+\delta|z|^2$ touches $w_\ell$ from above at points
  $z_\ell\to z_0$.  Because $\Hess_{\mathbb F} w_\ell\in U_\ell^\ast\A U_\ell$,
  \[
    \Hess_{\mathbb F}(\varphi+\delta|z|^2)(z_\ell)\in U_\ell^\ast\A U_\ell,
  \]
  hence $U_\ell\Hess_{\mathbb F}(\varphi+\delta|z|^2)(z_\ell)U_\ell^\ast\in\A$.
  Letting $\ell\to\infty$ and then $\delta\to0$, the closedness of $\A$
  gives $U\Hess_{\mathbb F}\varphi(z_0)U^\ast\in\A$. Hence $w$ is a viscosity subsolution of $\Hess_{\mathbb F} w \in U^\ast \Sigma U$. The supersolution case is analogous.

  \textbf{\ref{PropConsCond:P9}.}
  Let \(\varphi\) be a \(C^2\) function on \(\mathbb F^{n-1}\). We define a function on \(V\) by
  \[
    \tilde\varphi(z',\zeta)=\varphi(z').
  \]
  Because \(\tilde\varphi\) is independent of \(\zeta\) and the
  coordinates are chosen so that \(g\) is the standard inner product,
  its Hessian is
  \[
    \Hess_{\mathbb F}\tilde\varphi(z_0,0)
    = \begin{pmatrix}
        \Hess_{\mathbb F}\varphi(z_0) & 0 \\
        0 & 0
      \end{pmatrix}
    = q_N^\ast\bigl(\Hess_{\mathbb F}\varphi(z_0)\bigr),
  \]
  where \(q_N:V\to V/N\cong\mathbb F^{n-1}\) is the quotient map.

  Suppose \(\varphi\) touches \(v\) from above at \(z_0'\in\mathbb F^{n-1}\).
  Then \(\tilde\varphi\) touches \(w\) from above at \((z_0',0)\).
  Since \(w\) is a viscosity subsolution,
  \[
    \Hess_{\mathbb F}\tilde\varphi(z_0',0) \in \A .
  \]
  Hence \(q_N^\ast(\Hess_{\mathbb F}\varphi(z_0')) \in \A\), which
  means \(\Hess_{\mathbb F}\varphi(z_0') \in \A_{/N}\).  Thus
  \(v\) is a viscosity subsolution of
  \(\Hess_{\mathbb F}v\in\A_{/N}\).

  Now assume \(\Sigma_{/N}=\partial(\A_{/N})\).  Let \(\varphi\) touch
  \(v\) from below at \(z_0'\); then \(\tilde\varphi\) touches \(w\)
  from below at \((z_0',0)\).  Because \(w\) is a viscosity
  supersolution,
  \[
    \Hess_{\mathbb F}\tilde\varphi(z_0',0)
    \in \overline{\mathcal Q(V)\setminus\A}.
  \]
  If \(\Hess_{\mathbb F}\varphi(z_0')\) were an interior point
  of \(\A_{/N}\), then
  \(q_N^\ast(\Hess_{\mathbb F}\varphi(z_0'))\) would belong to
  \(\A\) by definition of \(\A_{/N}\), which together with the supersolution condition forces
  \[
    q_N^\ast(\Hess_{\mathbb F}\varphi(z_0'))
    \in \partial\A = \Sigma .
  \]
  Hence \(\Hess_{\mathbb F}\varphi(z_0') \in \Sigma_{/N}\).  By
  hypothesis \(\Sigma_{/N} = \partial(\A_{/N})\), contradicting that
  \(\Hess_{\mathbb F}\varphi(z_0') \in \operatorname{int}(\A_{/N})\).
  Therefore
  \[
    \Hess_{\mathbb F}\varphi(z_0')\notin\operatorname{int}(\A_{/N}),
  \]
  and \(v\) is a viscosity supersolution of
  \(\Hess_{\mathbb F}v\in\Sigma_{/N}\).  Combined with the subsolution
  property, \(v\) is a viscosity solution.
\end{proof}
\section{Proof Of \autoref{thm:main}}
In this section, we provide a proof of our main result. The forward direction is proved via a dimension reduction argument.
The following proposition is the key step.
It states that if the theorem already holds in dimensions \(\le n-1\), then any bounded, uniformly $\alpha$-H\"{o}lder continuous ($0 < \alpha \leq 1$) solution in dimension \(n\)
must be constant, provided the dichotomy in
\autoref{defn:Liouville-admissibility} propagates correctly to every
codimension-one subspace. The converse direction is established
separately at the end of this section.

\begin{proposition}\label{thm:mainproto}
  Let $0 < \alpha \leq 1$. Let \(\A\subset\mathcal Q(V)\) be admissible and set
  \(\Sigma=\partial\A\).  Let \(g\in\mathcal Q_{++}(V)\) be as in
  \ref{Cond:subharmonic} and assume that \(n=\dim_{\mathbb F}V\ge2\).
  Suppose that for every \(\mathbb F\)-linear subspace
  \(N\subset V\) with \(\dim_{\mathbb F}N=1\) one has either
  \begin{enumerate}[label=(\alph*)]
    \item \(\A_{/N}\) is contained in a terminal set, or
    \item \(\Sigma_{/N}=\partial(\A_{/N})\) and every bounded, uniformly $\alpha$-H\"{o}lder continuous viscosity solution \(u\colon V/N\to\mathbb R\) of
          \(\Hess_{\mathbb F}u\in\Sigma_{/N}\) is constant.
  \end{enumerate}
  Then every bounded, uniformly $\alpha$-H\"{o}lder continuous viscosity solution
  \(w\colon V\to\mathbb R\) of \(\Hess_{\mathbb F}w\in\Sigma\) is
  constant.
\end{proposition}
\begin{proof}
Identify \(V\) with \(\mathbb F^n\) so that the metric
      \(g\) from Definition~\ref{def:admissible} is the canonical
      Euclidean (if \(\mathbb F=\mathbb R\)) or Hermitian (if
      \(\mathbb F=\mathbb C\)) inner product. We will also use this metric $g$ to define the $[\cdot]_{0,\alpha;V}$ semi-norm.\par
      
  Suppose, for contradiction, that \(w\) is a non-constant bounded, uniformly $\alpha$-H\"{o}lder continuous viscosity solution.  Normalize \(w\) so that
  \(\sup_V w=1\) and \(\inf_V w=0\) by replacing \(w\) with
  \[
    \widetilde w(z)
    =
    C^{-2}\bigl(w(Cz)-\inf\nolimits_V w\bigr),
    \qquad
    C=\bigl(\sup\nolimits_V w-\inf\nolimits_V w\bigr)^{1/2}.
  \]
  The scaled function is still a bounded, uniformly $\alpha$-H\"{o}lder continuous viscosity solution of
  the same equation.

  Throughout the proof, balls, Lebesgue measure, and mollifications are
  taken with respect to the underlying real vector space. In the
  complex case, for a complex unit vector \(\xi\), the notation
  \(\nabla_\xi\) means the complex directional derivative
  \(\partial_\xi\).

  We split the argument into two mutually exclusive cases.

  \noindent\textbf{Case~1.}
  There exist a constant \(\beta>0\), a sequence
  \(\varepsilon_\ell\to0\), radii \(r_\ell\to\infty\), points \(z_\ell\in V\),
  and unit vectors \(\xi_\ell\in V\) such that
  \begin{equation}\label{eq:case1-cond}
    \begin{cases}
      \displaystyle
      \frac12\Bigl(
        \bigl[\tfrac12 [w]_{\varepsilon_\ell}^2\bigr]_{2r_\ell}(z_\ell)
        + [w]_\beta(z_\ell) - \frac43
      \Bigr) \ge w(z_\ell),\\[10pt]
      \displaystyle
      \lim_{\ell\to\infty}
      \int_{B_{r_\ell}(z_\ell)}
      |\nabla_{\xi_\ell}[w]_{\varepsilon_\ell}|^2\,dV = 0 .
    \end{cases}
  \end{equation}
  Here \([\,\cdot\,]_\varepsilon\) denotes mollification with the kernel
  defined in~\eqref{eq:moll-orig}.

  Let
  \[
    G_\ell(z)=z_\ell+U_\ell z,
  \]
  where \(U_\ell\in O(V)\) in the real case
  (resp.\ \(U_\ell\in\mathrm U(V)\) in the complex case) is chosen so that
  \(U_\ell e_n=\xi_\ell\).  Thus \(G_\ell(0)=z_\ell\), and the last coordinate
  direction in the new coordinates corresponds to \(\xi_\ell\).  After
  passing to a subsequence, we may assume that
  \[
    U_\ell\to U
  \]
  with \(U\in O(V)\) in the real case
  (resp.\ \(U\in\mathrm U(V)\) in the complex case).  Since the
  mollifier is radial, the conditions become
  \begin{equation}\label{eq:rotated-case1}
    \begin{cases}
      \displaystyle
      \frac12\Bigl(
        \bigl[\tfrac12 [w\circ G_\ell]_{\varepsilon_\ell}^2\bigr]_{2r_\ell}(0)
        + [w\circ G_\ell]_\beta(0) - \frac43
      \Bigr)
      \ge (w\circ G_\ell)(0),\\[10pt]
      \displaystyle
      \lim_{\ell\to\infty}
      \int_{B_{r_\ell}(0)}
      |\nabla_n [w\circ G_\ell]_{\varepsilon_\ell}|^2\,dV = 0 .
    \end{cases}
  \end{equation}
    In the real case, \(\nabla_n\) is the derivative in the \(n\)-th coordinate direction; in the complex case, it is the complex directional derivative in that direction.

  The family \(\{[w\circ G_\ell]_{\varepsilon_\ell}\}\) is uniformly bounded with $\sup_{\ell} [[w\circ G_\ell]_{\varepsilon_\ell}]_{0,\alpha;V} < \infty$.
  Hence, after passing to a subsequence, it converges locally uniformly to
  a bounded, uniformly $\alpha$-H\"{o}lder continuous function \(w_\star\).  Since \(w\) is uniformly $\alpha$-H\"{o}lder continuous and \(\varepsilon_\ell\to0\), 
  \begin{equation}\label{converg}
    w\circ G_\ell\to w_\star
  \end{equation}
  locally uniformly.  By stability under orthogonal (resp.\ unitary)
  transformations \ref{PropConsCond:P7} and under uniform limits
  \ref{PropConsCond:P8}, \(w_\star\) is a bounded, uniformly $\alpha$-H\"{o}lder continuous viscosity
  solution of
  \[
    \Hess_{\mathbb F} w_\star \in U^\ast\Sigma U .
  \]

  The gradient decay condition forces \(w_\star\) to be independent of
  the last coordinate. Indeed, by \eqref{eq:rotated-case1},
  \[
    \bigl\{\nabla_n [w\circ G_\ell]_{\varepsilon_\ell}\bigr\}
  \]
  is bounded in \(L^2_{\mathrm{loc}}(V)\).  Passing to a further
  subsequence, we may assume that
  \[
    \nabla_n [w\circ G_\ell]_{\varepsilon_\ell}
    \rightharpoonup g
    \qquad\text{weakly in }L^2_{\mathrm{loc}}(V)
  \]
  for some \(g\in L^2_{\mathrm{loc}}(V)\).   By \eqref{converg},
  %the distributional derivative of \(w_\star\) in the last coordinate direction coincides with \(g\).  Thus
  \[
    g=\nabla_n w_\star
    \qquad\text{in the sense of distributions.}
  \]
  Fix \(R>0\).  For all sufficiently large \(\ell\), we have
  \(B_R(0)\subset B_{r_\ell}(0)\).  Hence the second line of
  \eqref{eq:rotated-case1} gives
  \[
    \int_{B_R(0)}
    \bigl|\nabla_n [w\circ G_\ell]_{\varepsilon_\ell}\bigr|^2\,dV
    \longrightarrow 0 .
  \]
  By the weak lower semicontinuity of the $L^2$-norm,
  \[
    \int_{B_R(0)}|\nabla_n w_\star|^2\,dV
    \le
    \liminf_{\ell\to\infty}
    \int_{B_R(0)}
    \bigl|\nabla_n [w\circ G_\ell]_{\varepsilon_\ell}\bigr|^2\,dV
    =0 .
  \]
Therefore \(\nabla_n w_\star=0\) almost everywhere on \(B_R(0)\).
  Since \(R>0\) is arbitrary, \(\nabla_n w_\star=0\) almost everywhere on
  \(V\).  Consequently, \(w_\star\) is independent of the last
  coordinate.

  Identify \(V\) with \(\mathbb F^n\) so that \(g\) is the canonical
  metric, and let
  \(
    N = \{0\}^{n-1}\times\mathbb F \subset V
  \)
  be the one-dimensional subspace spanned by the last coordinate.
  Then \(w_\star \circ U^\ast\) is independent of the direction \(U N\),
  and induces a function \(\underline w \colon V/(U N) \to \mathbb R\).
  By \ref{PropConsCond:P7} and \ref{PropConsCond:P9},  \(\underline w\) is a viscosity
  subsolution of
  \begin{equation}\label{quot-sub}
    \Hess_{\mathbb F} \underline w \in \A_{/U N}.
  \end{equation}

  We now use the dichotomy in the statement of the proposition for the
  subspace \(U N\).
  \begin{itemize}
    \item If the first alternative holds, by the definition of a terminal
          set \autoref{def:liouvillecone}, \(\underline{w}\) is constant.

    \item If \(\Sigma_{/U N}=\partial(\A_{/U N})\) and all bounded  uniformly $\alpha$-H\"{o}lder continuous solutions in dimension \(n-1\) are constant, then the
          supersolution part of \ref{PropConsCond:P9} applies and
          \(\underline w\) is a viscosity solution of~\eqref{quot-sub}. We conclude
 by the inductive hypothesis 
 that \(\underline w\) is constant.
  \end{itemize}
Thus \(\underline w\equiv c\) for some \(c\in[0,1]\).   Hence $w_\star\equiv c$ on $V$.

  Now examine the first inequality in \eqref{eq:rotated-case1}.  Since
  \(0\le w\le1\), we have
  \[
    \bigl[\tfrac12 [w\circ G_\ell]_{\varepsilon_\ell}^2\bigr]_{2r_\ell}(0)
    \le \frac12 .
  \]
  Moreover, \(w\circ G_\ell\to c\) locally uniformly, so
  \[
    [w\circ G_\ell]_\beta(0)\to c .
  \]
  Hence, taking $\limsup_{\ell\to\infty}$ of the first line in \eqref{eq:rotated-case1}, we get \[ \frac12c-\frac{5}{12}= \frac12\Bigl(\frac12+c-\frac43\Bigr)\geq c  \] contradicting $c\in [0,1]$.  \textbf{Case~1} cannot occur.
      
  \noindent\textbf{Case~2.}
  Assume that \textbf{Case~1} does not occur. Then, by the negation of \textbf{Case~1}, for
every \(\beta>0\) there exist constants \(C_\beta>0\) and
\(c_\beta>0\) such that for all \(0<\varepsilon<C_\beta\), all
\(r>C_\beta^{-1}\), all \(z\in V\), and all unit vectors \(\xi\in V\),
  \begin{equation}\label{case2}
      \frac12\Bigl(
        \bigl[\tfrac12 [w]_\varepsilon^2\bigr]_{2r}(z)
        + [w]_\beta(z) - \frac43
      \Bigr) - w(z) \ge 0 \quad 
      \Longrightarrow\quad 
      \int_{B_r(z)}
      |\nabla_\xi [w]_\varepsilon|^2\,dV \ge c_\beta .
    \end{equation}

Since \(\inf_V w=0\), after recentering, we may assume $w(0)<\frac1{18}$.
By \autoref{cor:Cartan-density}, for the bounded non-negative
  subharmonic function \(w\), 
  \[
    \lim_{t\to\infty}[w]_t(0)
    =
    \lim_{t\to\infty}[w^2]_t(0)
    =
    1 .
  \]
  Choose \(\beta>0\) so large that
  \[
    [w]_\beta(0)>\frac{43}{45}.
  \]
  Let \(C_\beta,c_\beta\) be the constants given by the claim.  Choose
  \(0<\varepsilon<C_\beta\), and then choose \(r>C_\beta^{-1}\) so large that
  \[
    \bigl[\tfrac12 w^2\bigr]_{2r}(0)>\frac{22}{45}.
  \]
  Since \(w\) is subharmonic, the mollification inequality from
  \ref{PropConsCond:P5} gives
  \[
    [w]_\varepsilon\ge w .
  \]
  As \(0\le w\le1\), it follows that
  \begin{equation}\label{addd1}
    \bigl[\tfrac12 [w]_\varepsilon^2\bigr]_{2r}(0)
    \ge
    \bigl[\tfrac12 w^2\bigr]_{2r}(0)
    >
    \frac{22}{45}.
    \end{equation}
  Note that
  \[
    \frac1{18}
    =
    \frac12\Bigl(\frac{22}{45}+\frac{43}{45}-\frac43\Bigr),
  \]
  \eqref{addd1} implies
  \begin{equation}\label{w(0)}
    w(0)
    <
    \frac12\Bigl(
      \bigl[\tfrac12 [w]_\varepsilon^2\bigr]_{2r}(0)
      + [w]_\beta(0) - \frac43
    \Bigr).
    \end{equation}  
    
    Now choose a small parameter \(\tau_0>0\), to be fixed later, and define
  \begin{equation}\label{omega}
  \begin{aligned}
    v(z)
    &=
    \frac12\Bigl(
      \bigl[\tfrac12 [w]_\varepsilon^2\bigr]_{2r}(z)
      + [w]_\beta(z)
      - \frac43
      - \tau_0 |z|^2
    \Bigr),\\[4pt]
    \Omega
    &= \Bigl\{
      z\in V:
      w(z) < v(z)
    \Bigr\}.
  \end{aligned}
  \end{equation}
    Since the negative term \(-\frac{\tau_0}{2}|z|^2\) forces \(v(z)\to -\infty\) as \(|z|\to\infty\), while \(w\in[0,1]\), the set \(\Omega\) is bounded.
  Once we prove that \(v\) is a subsolution on \(\Omega\), the comparison principle
  \ref{PropConsCond:P6} applied to \(v\) and \(w\) on \(\Omega\) will yield a contradiction to \eqref{omega}.  Hence it remains to show that \(v\) is a subsolution on \(\Omega\).

  By \ref{PropConsCond:P3}, \ref{PropConsCond:P4}, and the mollification
  property \ref{PropConsCond:P5}, the functions
  \[
    [w]_\beta,\qquad
    [w]_\varepsilon,\qquad
    \frac12 [w]_\varepsilon^2,\qquad
    \bigl[\tfrac12 [w]_\varepsilon^2\bigr]_{2r}
  \]
  are smooth viscosity subsolutions of \(\Hess_{\mathbb F}(\cdot)\in\A\).
  It therefore suffices to verify that, for \(\tau_0>0\) sufficiently
  small depending only on \(c_\beta\), \(r\), and the mollifier \(\rho\),
  the function
  \[
    F(z)
    =
    \bigl[\tfrac12 [w]_\varepsilon^2\bigr]_{2r}(z)
    - \tau_0 |z|^2
  \]
  is also a smooth subsolution on \(\Omega\); indeed, by
  \ref{PropConsCond:P4}, we then have
  \begin{equation}
      \label{add101}
      v = \frac12(F + [w]_\beta - \frac43), 
  \end{equation}
  which is a subsolution on
  \(\Omega\).

  In the real case,
  \[
    \begin{aligned}
      \HessR(F)
      &=
      \HessR\Bigl(
        \bigl[\tfrac12 [w]_\varepsilon^2\bigr]_{2r}
      \Bigr)
      -2\tau_0 I_n  \\
      &=
      \bigl[
        [w]_\varepsilon \HessR([w]_\varepsilon)
        +
        \nabla [w]_\varepsilon\otimes\nabla [w]_\varepsilon
      \bigr]_{2r}
      -2\tau_0 I_n .
    \end{aligned}
  \]
  In the complex case,
  \[
    \begin{aligned}
      \HessC(F)
      &=
      \HessC\Bigl(
        \bigl[\tfrac12 [w]_\varepsilon^2\bigr]_{2r}
      \Bigr)
      -\tau_0 I_n  \\
      &=
      \bigl[
        [w]_\varepsilon \HessC([w]_\varepsilon)
        +
        \nabla [w]_\varepsilon\otimes\overline{\nabla} [w]_\varepsilon
      \bigr]_{2r}
      -\tau_0 I_n .
    \end{aligned}
  \]
  In both settings the gradient term is positive semidefinite. For every unit vector \(\xi\in V\), we estimate
  \begin{equation}\label{add6}
    \begin{aligned}
      \bigl[|\nabla_\xi [w]_\varepsilon|^2\bigr]_{2r}(z)
      &=
      \int_V
      |\nabla_\xi [w]_\varepsilon(z-2ry)|^2
      \rho(y)\,dV(y)  \\
      &\ge
      \int_{B_{1/2}(0)}
      |\nabla_\xi [w]_\varepsilon(z-2ry)|^2
      \rho(y)\,dV(y)  \\
      &\ge
      (2r)^{-\dim_{\mathbb R}V}
      \Bigl(\inf_{B_{1/2}(0)}\rho\Bigr)
      \int_{B_r(z)}
      |\nabla_\xi [w]_\varepsilon|^2\,dV .
    \end{aligned}
    \end{equation}
  If \(z\in\Omega\), the definition of $\Omega$ \eqref{omega} implies
  the antecedent in the \textbf{Case~2} claim.  Therefore,
  \begin{equation}\label{add7}
    \int_{B_r(z)}
    |\nabla_\xi [w]_\varepsilon|^2\,dV
    \ge c_\beta .
    \end{equation}
  By \eqref{add6} and \eqref{add7},
  \[
    \bigl[|\nabla_\xi [w]_\varepsilon|^2\bigr]_{2r}(z)
    \ge C(r,\rho)c_\beta ,
  \]
  where
  \[
    C(r,\rho)
    =
    (2r)^{-\dim_{\mathbb R}V}\inf_{B_{1/2}(0)}\rho
    >0 .
  \]
  
  Choose \(\tau_0>0\) sufficiently small so that the averaged gradient term dominates the negative quadratic contribution in the Hessian of $F$.  Thus, on
  \(\Omega\),
  \begin{equation}\label{F}
    \Hess_{\mathbb F}(F)
    \ge
    \bigl[
      [w]_\varepsilon\Hess_{\mathbb F}([w]_\varepsilon)
    \bigr]_{2r}.
    \end{equation}
  Since \(0\le [w]_\varepsilon\le1\) and
  \(\Hess_{\mathbb F}([w]_\varepsilon)\in\A\), the matrix
  \begin{equation}\label{F1}
      [w]_\varepsilon\Hess_{\mathbb F}([w]_\varepsilon)\in \A
  \end{equation}
  by \ref{Cond:origin} and \ref{Cond:convex}.  Its
  mollification is a convex combination of elements of \(\A\), hence
  still belongs to \(\A\). From \eqref{F}, \eqref{F1} and \ref{Cond:up},
  \[
    \Hess_{\mathbb F}(F)\in\A
    \qquad\text{on }\Omega .
  \]
Therefore, \(F\) is a smooth subsolution on \(\Omega\); and by \eqref{add101}, so is $v$. By \eqref{omega}  , $v = w$ on $\partial \Omega$. The comparison principle \ref{PropConsCond:P6} applied to \(v\) and \(w\) on \(\Omega\) yields \(w \ge v\) on \(\Omega\).
 By \eqref{omega} again, \(\Omega=\varnothing\). However,  \eqref{w(0)}  implies that \(0\in\Omega\). We have a contradiction.  Hence \textbf{Case~2} is also impossible.

  Both cases lead to contradictions.  Hence the original bounded, uniformly $\alpha$-H\"{o}lder continuous viscosity solution \(w\) must be constant.
\end{proof}
We first prove the forward direction of \autoref{thm:main}. With \autoref{thm:mainproto} in hand, the argument is a straightforward induction.
\begin{theorem}\label{thm:mainforward}
Let $0 < \alpha \leq 1$. Let
\(
\A\subset\mathcal Q(V)
\)
be Liouville admissible and let
\(
\Sigma=\partial\A.
\)
If \(u:V\to\mathbb R\) is a bounded, uniformly $\alpha$-H\"{o}lder continuous viscosity solution of
\[
\Hess_{\mathbb F}u\in\Sigma,
\]
then \(u\) is constant.
\end{theorem}
\begin{proof}
 We proceed by induction on the dimension \(n=\dim_{\mathbb F}V\). 

  If $n = 1$,  by \ref{Cond:subharmonic}, any viscosity subsolution of \(\Hess_{\mathbb F}u\in\A\) is subharmonic.  A bounded viscosity
  solution of \(\Hess_{\mathbb F}u\in\Sigma\) is therefore a bounded subharmonic function on \(\mathbb R\) or \(\mathbb C\), which has to  be constant.  Hence the theorem holds in dimension one.

  Now, assume the theorem holds for all Liouville admissible sets in spaces of dimensions
  \(\le n-1\).  Let \(\A\subset\mathcal Q(V)\) be Liouville
  admissible with \(\dim_{\mathbb F}V=n\ge2\), let
  \(g\in\mathcal Q_{++}(V)\) be as in \ref{Cond:subharmonic}, and set
  \(\Sigma=\partial\A\).  Let \(u\colon V\to\mathbb R\) be a bounded, uniformly $\alpha$-H\"{o}lder continuous viscosity solution of \eqref{equation}.
We verify the hypotheses of \autoref{thm:mainproto}.  Let
  \(N\subset V\) be any one-dimensional \(\mathbb F\)-linear
  subspace.  By Liouville admissibility, one has either
  \begin{enumerate}[label=(\alph*)]
    \item\label{11} \(\A_{/N}\) is contained in a terminal set, or
    \item\label{22} \(\Sigma_{/N}=\partial(\A_{/N})\).
  \end{enumerate}
  Case~\ref{11} directly fulfills the first condition of
  \autoref{thm:mainproto}.
In case~\ref{22}, \(\A_{/N}\) is Liouville admissible by
  \autoref{lem:Liouvilleadmissibilityofquotient}.  The quotient space
  \(V/N\) has dimension \(n-1\).  By the induction hypothesis
  applied to \(\A_{/N}\),  the second alternative of \autoref{thm:mainproto} holds.

  In summary, for every one-dimensional subspace \(N\subset V\), we have verified the
  hypotheses of \autoref{thm:mainproto}. Therefore,  \autoref{thm:mainproto} implies that $u$ is constant. This completes the
  induction.
\end{proof}
We now prove the converse direction. 
\begin{proposition}\label{prop:implicationwithoutholder}
    Let $\A \subset \mathcal{Q}(V)$ be admissible. If every bounded
    entire viscosity solution of $\Hess_{\mathbb F} u \in \Sigma$ is
    constant, then $\A$ is Liouville admissible.
\end{proposition}

\begin{proof}
    Suppose that $\A$ is not Liouville admissible. By
    \autoref{prop:boundary-dichotomy-reduction}, there exists a proper
    nontrivial $\mathbb{F}$-linear subspace $N \subset V$ such that
    $\Sigma_{/N} = \A_{/N}$ and $\A_{/N}$ is not contained in any
    terminal set. In particular, $\A_{/N}$ is not terminal; hence there
    exists a nonconstant continuous viscosity subsolution $u$ of
    $\Hess_{\mathbb F} u \in \A_{/N}$ that is bounded above on $V/N$.

    Since the maximum of two viscosity subsolutions is again a viscosity subsolution, we may replace \(u\) by \(\max\{u,-C\}\), where \(C>0\) is
chosen sufficiently large. Therefore, we may assume that $u$ is bounded and
nonconstant.

    Let $h_{V/N}$ be the metric on $V/N$ provided by
    \autoref{lem:admissibilityofquotient} for the admissible set
    $\A_{/N}$. By \ref{PropConsCond:P5}, the mollification
    $[u]_\varepsilon$ with respect to $h_{V/N}$ remains a smooth
    subsolution of
    \[
        \Hess_{\mathbb F} [u]_\varepsilon \in \A_{/N}.
    \]
    Since $[u]_\varepsilon \to u$ locally uniformly as
    $\varepsilon \to 0$, we may choose $\varepsilon>0$ sufficiently small
    so that $[u]_\varepsilon$ is still nonconstant. Since
    $\A_{/N}=\Sigma_{/N}$, the pullback
    $q_N^*[u]_\varepsilon$ satisfies
    \[
        \Hess_{\mathbb F}(q_N^*[u]_\varepsilon)\in \Sigma .
    \]
    Moreover, $q_N^*[u]_\varepsilon$ is entire, bounded, and
    nonconstant, contradicting the hypothesis on $\A$.
\end{proof}

\begin{proof}[Proof of Theorem~\ref{thm:main}]
The proof of the preceding \autoref{prop:implicationwithoutholder} gives more than a bounded
counterexample. Indeed, after the truncation \(u\mapsto \max\{u,-C\}\),
the subsolution is bounded. Hence, for each fixed \(\varepsilon>0\), its
mollification \([u]_\varepsilon\) is bounded and globally Lipschitz on
\(V/N\). Choosing \(\varepsilon>0\) sufficiently small, it remains
nonconstant. Therefore the pullback \(q_N^*([u]_\varepsilon)\) is an entire, bounded,
globally Lipschitz, hence globally \(C^{0,\alpha}\), nonconstant
viscosity solution of
\(\Hess_{\mathbb F}u\in\Sigma\). This contradicts the assumed
\(C^{0,\alpha}\)-Liouville property.

This proves the converse implication. Combined with the forward implication proved above in \autoref{thm:mainforward}, this completes the proof of Theorem~\ref{thm:main}.
\end{proof}

\appendix

\section{Maximal Spectral Terminal Sets}

\label{Appendix1}The existence of terminal sets has been discussed in Proposition~\ref{prop:positive-cone-terminal}. However, in the real and spectral case,  we have more flexible constructions, which are linked to the classical works of Harvey and Lawson~\autocite{MR3884610}.

\begin{definition}\label{def:real-log-cone}
Let \(W\) be a real vector space of dimension \(m\), and let
\(
g\in\SSS_{++}(W)
\)
be a Euclidean metric. 
If \(m=1\), define
\[
\LLL_g(W)
=
\{A\in\SSS(W):\Tr_g A\geq 0\} .\]

If \(m\ge2\), define
\[
\LLL_g(W)=\Bigl\{A\in\SSS(W):(\Tr_g A)\,g+(m-2)A\in\SSS_+(W) \Bigr\}.
\]
We call \(\LLL_g(W)\) the \emph{real logarithmic cone} associated with
\(g\).
\end{definition}
  The real logarithmic cone $\LLL_g(W)$ is the cone defined in \autocite[Example 4.3]{MR3884610} with Riesz characteristic $2$, which depends on the chosen metric. The complex analogue of Riesz characteristic $2$ is simply the space of positive semidefinite Hermitian forms, and is, in particular, metric-independent.

\begin{proposition}
Let \((W,g)\) be a finite-dimensional real Euclidean vector space. Then the logarithmic cone $
\LLL_g(W)$
is a terminal set. \label{logterminal}
\end{proposition}

\begin{proof}
Let \(m=\dim W\). After choosing a \(g\)-orthonormal basis, we may
assume
\[
(W,g)=(\R^m,I_m).
\]

If \(m=1\) and $\HessR (u)\in \LLL_g(\R)$, then $u$ is subharmonic and therefore convex.
 A convex function on \(\R\) which is
bounded above is constant.

Assume \(m\ge2\). Let $u$ be bounded from above and
be a viscosity subsolution of
\[
\HessR u\in\LLL_g(\R^m).
\]
Set
\(
M=\sup_{\R^m}u.
\) For \(\mathbf x\neq0\),
\[
\HessR\log|\mathbf x|
=
\frac1{|\mathbf x|^2}
\left(
I_m-2\frac{\mathbf x\otimes\mathbf x}{|\mathbf x|^2}
\right).
\]
Its eigenvalues are
\[
-\frac1{|\mathbf x|^2},
\frac1{|\mathbf x|^2},\dots,
\frac1{|\mathbf x|^2}.
\]
It follows that
\[
\HessR\log|\mathbf x|
\in
\partial\LLL_g(\R^m)
\qquad
(\mathbf x\neq0).
\]

Fix \(\mathbf x_0\in\R^m\). For \(0<r<R\), let
\[
A_{r,R}
=
B_R(\mathbf x_0)\setminus\overline{B_r(\mathbf x_0)},
\]
and define
\[
a_r
=
\sup_{\partial B_r(\mathbf x_0)}u.
\]

Consider the test function
\[
h_{r,R}(\mathbf x)
=
a_r
+
(M-a_r)
\frac{\log(|\mathbf x-\mathbf x_0|/r)}
{\log(R/r)}.
\]
One has $\HessR h_{r,R}\in\partial\LLL_g(\R^m)$ on $A_{r,R}$.

Moreover,
\[
h_{r,R}=a_r
\quad
\text{on }\partial B_r(\mathbf x_0),
\qquad
h_{r,R}=M
\quad
\text{on }\partial B_R(\mathbf x_0).
\]
Since
\[
u\le a_r
\quad
\text{on }\partial B_r(\mathbf x_0),
\qquad
u\le M
\quad
\text{on }\partial B_R(\mathbf x_0),
\]
comparison for the convex elliptic subequation \(\LLL_g\) yields
\[
u\le h_{r,R}
\qquad
\text{on }A_{r,R}.
\]

Fix \(\mathbf y\neq\mathbf x_0\), and choose
\(
0<r<|\mathbf y-\mathbf x_0|.
\)
For sufficiently large \(R\),
\(
\mathbf y\in A_{r,R},
\)
hence,
\[
u(\mathbf y)
\le
a_r
+
(M-a_r)
\frac{\log(|\mathbf y-\mathbf x_0|/r)}
{\log(R/r)}.
\]
Letting \(R\to\infty\) gives
\[
u(\mathbf y)\le a_r.
\]
Finally, letting \(r\to0\) and using upper semicontinuity,
\[
u(\mathbf y)\le u(\mathbf x_0).
\]
The interchange of \(\mathbf x_0\) and \(\mathbf y\) gives equality. Hence
\(u\) is constant.
\end{proof}

\begin{remark}
Terminal sets are not unique. In the real case, different choices of
the Euclidean metric $g$
generally give different logarithmic cones
\(
\LLL_g(W)\subset\SSS(W).
\)
\end{remark}

By Lemma~\ref{lem:terminal-stability}, there exists a partial order on terminal sets by inclusion. In the spectral case,  we obtain the following maximal property. See also   \autocite[Proposition~13.10]{MR3884610} for a classical interpretation using Riesz characteristic.
\begin{proposition}
\label{prop:maximal-spectral-terminal-sets}
Let \(V\) be a finite-dimensional real Euclidean or complex Hermitian
vector space, and let \(\TT\subset\mathcal Q(V)\) be a spectral terminal
set.

If \(\mathbb F=\mathbb R\), then
\begin{equation}
\label{eq:real-spectral-terminal-contained-log}
    \TT\subset\LLL_g(V).
\end{equation}
If \(\mathbb F=\mathbb C\), then
\begin{equation}
\label{eq:complex-spectral-terminal-equals-positive}
    \TT=\mathcal H_+(V).
\end{equation}
Consequently, \(\LLL_g(V)\) is the maximal real spectral terminal set,
while \(\mathcal H_+(V)\) is the unique complex spectral terminal set.
\end{proposition}

\begin{proof}
Suppose first that \(\mathbb F=\mathbb R\) and
\(n=\dim_{\mathbb R}V\ge3\). If \(\TT\not\subset\LLL_g(V)\), then some
\(A\in\TT\) has ordered eigenvalues
$\lambda_1\le\cdots\le\lambda_n$ satisfying
\[
    \sum_{i=1}^n\lambda_i+(n-2)\lambda_1<0.
\]

By spectral invariance and convexity, averaging over permutations of
the last \(n-1\) coordinates yields
$(a,b,\ldots,b)\in\TT$ with \(a+b<0\). Choose \(\tau\) with $\max\{0,-b\}<\tau<-a-b$. Set
\[
    B:=b+\tau>0,
    \qquad
    s:=\frac{a}{B}<-1 .
\]
By ellipticity,
\[
    (a,B,\ldots,B)\in\TT .
\]
Choose \(q>1\) with \(1-2q\ge s\), and set
\[
    w(x)=-(1+|x|^2)^{1-q}.
\]
The eigenvalues of \(\Hess_{\mathbb R}w\) are
\[
    \lambda^{\mathbb{R}}_T(r)
    \left(
        1-\frac{2qr^2}{1+r^2},
        1,\ldots,1
    \right),
    \qquad
    \lambda^{\mathbb{R}}_T(r)=2(q-1)(1+r^2)^{-q}.
\]
Set
\[
    \theta(r):=1-\frac{2qr^2}{1+r^2}\ge s
    \qquad\text{for all }r\ge0.
\]
Choose \(\varepsilon>0\) so small that
\[
    0\le \frac{\varepsilon\lambda^{\mathbb{R}}_T(r)}{B}\le1
    \qquad\text{for all }r\ge0.
\]
Then, for \(u=\varepsilon w\), the eigenvalues of
\(\Hess_{\mathbb R}u\) are
\[
    \frac{\varepsilon\lambda^{\mathbb{R}}_T(r)}{B}(a,B,\ldots,B)
    +
    \varepsilon\lambda^{\mathbb{R}}_T(r)
    \left(
        \theta(r)-s,
        0,\ldots,0
    \right).
\]
The first term lies in \(\TT\), because \(0\in\TT\), \(\TT\) is convex,
and \((a,B,\ldots,B)\in\TT\). The second term is positive semidefinite.
By ellipticity,
\[
    \Hess_{\mathbb R}u\in\TT.
\]
But \(u\) is smooth, bounded above, and nonconstant. This contradicts
terminality. Hence
\eqref{eq:real-spectral-terminal-contained-log} holds for \(n\ge3\). The cases
\(\dim_{\mathbb R}V=1,2\) follow immediately from the definition of
\(\LLL_g(V)\).

Now suppose \(\mathbb F=\mathbb C\). Let \(n=\dim_{\mathbb C}V\).
Since \(\TT\) is admissible,
\(0\in\TT\) and \(\TT+\mathcal H_+(V)\subset\TT\). Therefore
\[
    \mathcal H_+(V)\subset\TT.
\]
It remains to prove the reverse inclusion. Suppose
\(\TT\not\subset\mathcal H_+(V)\). Then some \(A\in\TT\) has a negative
eigenvalue. Let its ordered eigenvalues be
\[
    \lambda_1\le\cdots\le\lambda_n,
    \qquad
    \lambda_1<0.
\]
By spectral invariance and convexity, averaging over permutations of
\(\lambda_2,\ldots,\lambda_n\) gives
\[
    (a,b,\ldots,b)\in\TT,
    \qquad
    a=\lambda_1<0.
\]
Choose \(\tau>0\) so that \(B:=b+\tau>0\). By ellipticity,
\[
    (a,B,\ldots,B)\in\TT.
\]
Set \(s:=a/B<0\), choose \(\alpha>0\) with \(-\alpha\ge s\), and define
\[
    w(z)=-(1+|z|^2)^{-\alpha}.
\]
Writing \(t=|z|^2\), the eigenvalues of
\(\Hess_{\mathbb C}w=(w_{i\bar j})\) are
\[
    \lambda^{\mathbb{C}}_T(t)
    \left(
        \frac{1-\alpha t}{1+t},
        1,\ldots,1
    \right),
    \qquad
    \lambda^{\mathbb{C}}_T(t)=\alpha(1+t)^{-\alpha-1}.
\]
The first component satisfies
\[
    \frac{1-\alpha t}{1+t}\ge -\alpha\ge s.
\]
Choose \(\varepsilon>0\) so small that
\[
    0\le \frac{\varepsilon\lambda^{\mathbb{C}}_T(t)}{B}\le1
    \qquad\text{for all }t\ge0.
\]
Then, for \(u=\varepsilon w\), the eigenvalues of
\(\Hess_{\mathbb C}u\) are
\[
    \frac{\varepsilon\lambda^{\mathbb{C}}_T(t)}{B}(a,B,\ldots,B)
    +
    \varepsilon\lambda^{\mathbb{C}}_T(t)
    \left(
        \frac{1-\alpha t}{1+t}-s,
        0,\ldots,0
    \right).
\]
The first term lies in \(\TT\), because \(0\in\TT\), \(\TT\) is convex,
and \((a,B,\ldots,B)\in\TT\). The second term is positive semidefinite.
By ellipticity,
\[
    \Hess_{\mathbb C}u\in\TT.
\]
But \(u\) is smooth, bounded above, and nonconstant. This contradicts
terminality. Hence \(\TT\subset\mathcal H_+(V)\), and therefore
\eqref{eq:complex-spectral-terminal-equals-positive} follows.

The cone \(\LLL_g(V)\) is terminal in the real case by
Proposition~\ref{logterminal}, while \(\mathcal H_+(V)\) is terminal in
the complex case by Proposition~\ref{prop:positive-cone-terminal}. The
stated maximality and uniqueness follow.
\end{proof}

\section{An Anisotropic Example}
\begin{lemma}\label{lem:terminal-product}
    Let $\mathbb F \in \{\mathbb R, \mathbb C\}$, let $V, W$ be vector 
    spaces over $\mathbb F$, and let $\A_V \subset \mathcal{Q}(V)$ and $\A_W \subset \mathcal{Q}(W)$ be terminal sets. Then
    \begin{equation}
        \A_{V \oplus W} = \{A \in \mathcal{Q}(V \oplus W): A|_{V} \in \A_V,\ A|_{W} \in \A_W\}
    \end{equation}
    is a terminal set in $\mathcal{Q}(V \oplus W)$, where \(A|_V\) and \(A|_W\) denote the restrictions of \(A\) to
\(V\oplus\{0\}\) and \(\{0\}\oplus W\), respectively.
\end{lemma}
\begin{proof}
    We first verify that $\A_{V \oplus W}$ is admissible. Closedness, convexity, and $0 \in \A_{V \oplus W}$ follow immediately from the corresponding properties of $\A_V$ and $\A_W$, because all conditions are checked componentwise. For $A \in \A_{V \oplus W}$ and $P \in \mathcal{Q}_+(V \oplus W)$, the restrictions satisfy $(A+P)|_V = A|_V + P|_V \in \A_V$ and $(A+P)|_W = A|_W + P|_W \in \A_W$, since $P|_V$ and $P|_W$ are positive semidefinite and each $\A_V, \A_W$ is admissible. Hence $A+P \in \A_{V \oplus W}$.

    Let $h_V$ and $h_W$ be metrics on $V$ and $W$ such that $\Tr_{h_V} A \ge 0$ for all $A \in \A_V$ and $\Tr_{h_W} B \ge 0$ for all $B \in \A_W$.  Define $h_{V \oplus W} = h_V \oplus h_W$. For any $A \in \A_{V \oplus W}$,
    \[
        \Tr_{h_{V \oplus W}} A = \Tr_{h_V}(A|_V) + \Tr_{h_W}(A|_W) \ge 0 .
    \]
    Thus $\A_{V \oplus W}$ is admissible.

    Now let $u$ be a bounded-above continuous viscosity subsolution of $\Hess_\mathbb{F} u \in \A_{V \oplus W}$. By \ref{PropConsCond:P5}, the mollification $u_\varepsilon$ with respect to $h_{V \oplus W}$ remains a subsolution of $\Hess_\mathbb{F} u_\varepsilon \in \A_{V \oplus W}$.

    For each fixed $y \in W$, consider the slice $v_y(x) = u_\varepsilon(x, y)$. Its Hessian satisfies $\Hess_\mathbb{F} v_y(x) = (\Hess_\mathbb{F} u_\varepsilon(x, y))|_V \in \A_V$, because $\Hess_\mathbb{F} u_\varepsilon(x, y) \in \A_{V \oplus W}$ and the restriction to $V$ of any element of $\A_{V \oplus W}$ belongs to $\A_V$ by definition.  Thus $v_y$ is a bounded-above $\A_V$-subsolution on $V$. Since $\A_V$ is terminal, $v_y$ is constant. Hence $u_\varepsilon$ is independent of the $V$-variable.

    A symmetric argument shows that $u_\varepsilon$ is independent of the $W$-variable. Therefore $u_\varepsilon$ is constant on $V \oplus W$ for every sufficiently small $\varepsilon > 0$. Letting $\varepsilon \to 0$, the uniform convergence of $u_\varepsilon$ to $u$ implies that $u$ is constant.  Hence $\A_{V \oplus W}$ is terminal.
\end{proof}
\begin{example}[Products of logarithmic cones]\label{ex:product-log-cones}
For $k \ge 2$, let $W_1,\dots,W_k$ be finite-dimensional real vector spaces,
each endowed with a Euclidean metric $g_i$ and with
$\dim W_k \ge 3$.  Form the orthogonal direct sum
\[
V = W_1 \oplus \cdots \oplus W_k,
\qquad
g_V = g_1 \oplus \cdots \oplus g_k .
\]
For each $i$, let $\LLL_{g_i}(W_i)$ be the real logarithmic cone
on $W_i$ with respect to $g_i$ (Definition~\ref{def:real-log-cone}),
and define
\begin{equation}\label{eq:product-terminal}
\A = \bigl\{ A \in \mathcal{Q}(V):
A|_{W_i} \in \LLL_{g_i}(W_i) \;\text{for all } i \bigr\}.
\end{equation}
No condition is imposed on the off-diagonal blocks.

By Proposition~\ref{logterminal}, each $\LLL_{g_i}(W_i)$ is a terminal
set.  Lemma~\ref{lem:terminal-product} implies that $\A$ is terminal
and, in particular, Liouville admissible.  Moreover, $\A$ is not
spectral and not invariant under the full orthogonal group of $V$,
since its definition depends on the distinguished splitting
$V = W_1 \oplus \cdots \oplus W_k$ and on the choice of each metric
$g_i$.  A rotation that mixes different subspaces $W_i$ and $W_j$
does not preserve the block-wise conditions
in~\eqref{eq:product-terminal}, so $\A$ is anisotropic.
\end{example}\begin{example}[Failure of the boundary equality for a quotient]\label{ex:boundary-equality-failure}
Consider the same setting as in Example~\ref{ex:product-log-cones}
with $k=2$. Write $V = W \oplus E$, where
$W = W_1$, $E = W_2$, $\dim W \ge 3$, $\dim E \ge 3$.
Let $\A$ be as in \eqref{eq:product-terminal} and set
$\Sigma = \partial\A$.

The quotient reduction along $E$ gives
$\A_{/E} = \LLL_{g_W}(W)$.
Since $0 \in \partial\LLL_{g_E}(E)$, every lift
$q_E^* B = \begin{pmatrix} B & 0 \\ 0 & 0 \end{pmatrix}$ with
$B \in \LLL_{g_W}(W)$ lies in $\Sigma$, whence
$\Sigma_{/E} = \A_{/E}$.
Because $\LLL_{g_W}(W)$ has nonempty interior,
$\partial(\A_{/E})$ is a proper subset of $\A_{/E}$, so the
second alternative of Liouville admissibility fails for this
quotient.  Nevertheless, $\A$ is terminal and hence Liouville
admissible, satisfying the first alternative trivially.

This example illustrates why enlarging the class of terminal
sets beyond the positive semidefinite cone $\mathcal{Q}_+(W)$
is essential for the theory.  Had we restricted the definition
of Liouville admissibility by allowing only
$\mathcal{Q}_+(W)$ as a terminal set, the quotient
$\A_{/E} = \LLL_{g_W}(W)$ would violate both alternatives:
it is not contained in $\mathcal{Q}_+(W)$, and the boundary
equality fails.  By admitting the logarithmic cone as a
terminal set, we retain Liouville admissibility for $\A$ and
thus obtain a Liouville theorem for a class of admissible sets
that are neither spectral nor isotropic.
\end{example}

\printbibliography

\end{document}